\documentclass[11pt,reqno]{amsart}

\usepackage[T1]{fontenc}
\usepackage[utf8]{inputenc}
\usepackage{lmodern}
\usepackage{microtype}
\usepackage[margin=1.08in]{geometry}
\usepackage{amsmath,amssymb,mathtools}
\usepackage{booktabs}
\usepackage{array}
\usepackage{needspace}
\usepackage{xcolor}
\usepackage{graphicx}
\usepackage{tikz}
\usetikzlibrary{decorations.pathreplacing,calc,positioning}
\newlength{\permcell}\newlength{\permdot} 
\usepackage{aliascnt}
\usepackage{hyperref}
\usepackage{xurl}

\usepackage[capitalise,nameinlink,noabbrev]{cleveref}

\definecolor{pairblue}{RGB}{37,99,154}
\definecolor{tailorange}{RGB}{205,105,31}
\definecolor{darkgreen}{RGB}{39,108,72}

\hypersetup{
  colorlinks=true,
  linkcolor=pairblue,
  citecolor=darkgreen,
  urlcolor=tailorange,
  pdftitle={Protected tails and polynomial-time enumeration of permutations
    avoiding a direct sum of an increasing pattern and 231},
  pdfauthor={Henning Arn\'{o}r Skeggi \'{U}lfarsson},
  pdflang={en},
  pdfkeywords={permutation pattern, pattern avoidance, polynomial-time
    algorithm, Wilfian formula}
}

\newtheorem{theorem}{Theorem}[section]
\newaliascnt{proposition}{theorem}
\newtheorem{proposition}[proposition]{Proposition}
\aliascntresetthe{proposition}
\newaliascnt{lemma}{theorem}
\newtheorem{lemma}[lemma]{Lemma}
\aliascntresetthe{lemma}
\newaliascnt{corollary}{theorem}
\newtheorem{corollary}[corollary]{Corollary}
\aliascntresetthe{corollary}
\newaliascnt{conjecture}{theorem}
\newtheorem{conjecture}[conjecture]{Conjecture}
\aliascntresetthe{conjecture}
\theoremstyle{definition}
\newaliascnt{definition}{theorem}
\newtheorem{definition}[definition]{Definition}
\aliascntresetthe{definition}
\newaliascnt{example}{theorem}
\newtheorem{example}[example]{Example}
\aliascntresetthe{example}
\theoremstyle{remark}
\newaliascnt{remark}{theorem}
\newtheorem{remark}[remark]{Remark}
\aliascntresetthe{remark}

\newcommand{\Av}{\operatorname{Av}}

\newcommand{\nz}{\operatorname{nz}}
\newcommand{\supp}{\operatorname{supp}}

\newcommand{\N}{\mathbb N}
\newcommand{\emptyword}{\varepsilon}
\newcommand{\one}{\mathbf 1}

\title[Protected tails and polynomial-time enumeration]
      {Protected tails and polynomial-time enumeration of permutations
       avoiding a direct sum of an increasing pattern and 231}
\author{Henning Arn\'{o}r Skeggi \'{U}lfarsson}
\address{Department of Computer Science, Reykjavik University,
Menntavegur 1, 102 Reykjavik, Iceland}
\email{henningu@ru.is}
\date{September 2026}

\begin{document}
\raggedbottom

\begin{abstract}
We give an algorithm counting the permutations that avoid a fixed
pattern of the following form: the direct sum of an increasing pattern and
231.  The first members of the family are 1342 and 12453.  For each member
the algorithm uses polynomially many operations and stored
integers, with degrees that grow linearly in the length of the pattern.  It
comes from a recurrence that reads a permutation from left to right and
records the constraints that the letters read so far impose on those still
unread.  This recurrence has exponentially many states, but part of each
state is protected: later steps carry it along unchanged and do not depend
on it, and factoring the protected part out leaves a dynamic program of
polynomial size.  For 12453 a translation symmetry sharpens the bounds to
degree seven for the operations and degree four for the storage.  We also
compute the number of 12453-avoiding permutations of every length up to 150.
The previously published series, due to Biers-Ariel (2019), reached length
38.  We also give a sampler of uniformly random avoiders.  A
floating-point implementation of it, proved to be within total variation
distance $3.5\cdot10^{-5}$ of uniform for ideal random bits, draws the one
million
12453-avoiding permutations of length 300 shown in a heatmap.  The literal
and kernel recurrences for 1342 and 12453 are verified in the Lean~4 proof
assistant.
\end{abstract}


\maketitle

\section{Introduction}

A permutation $\pi=\pi_1\cdots\pi_n$ of
$[n]=\{1,\ldots,n\}$ \emph{contains} another permutation
$\tau$ of $[k]$ if there are indices
$1\leq i_1<\cdots<i_k\leq n$ for which
$\pi_{i_1},\ldots,\pi_{i_k}$ has the same relative order as $\tau$. In this
case we say that $\tau$ is a \emph{pattern} in $\pi$.
Otherwise $\pi$ \emph{avoids} $\tau$.  For example, the entries
$2,5,7,9,6$ in $251796834$ have relative ranks
$1,2,4,5,3$, and therefore form an occurrence of $12453$.  We write
$\Av_n(\tau)$ for the set of $\tau$-avoiding
permutations of $[n]$, and $\Av(\tau)=\bigcup_{n\geq0}\Av_n(\tau)$.  For a
set $B$ of patterns, $\Av(B)$ is the set of permutations that avoid every
member of $B$, and $\Av(B)$ is \emph{finitely based} if $B$ is finite.
Throughout, $\N=\{0,1,2,\ldots\}$.

The basic enumerative problem is to determine
\[
  a_n(\tau)=\lvert\Av_n(\tau)\rvert.
\]
By \emph{polynomial-time enumeration} we mean computing this integer, not
listing the avoiding permutations.  In other words, we seek a \emph{Wilfian formula} in the sense of
Wilf~\cite{Wilf1982} and Zeilberger~\cite{Zeilberger1998}.  For a fixed pattern,
the theorem of Marcus and Tardos~\cite{MarcusTardos2004} gives an
exponential upper bound $a_n(\tau)\leq c(\tau)^n$ on the counting sequence.
Arratia~\cite{Arratia1999} showed that $a_n(\tau)^{1/n}$ has a limit in
$[1,\infty]$.  The limit, which is finite by this bound, is called the \emph{Stanley--Wilf
limit} of $\tau$, or the \emph{growth rate} of $\Av(\tau)$.  The task is to obtain $a_n(\tau)$ in at most $P(n)$
steps, where $P$ is a polynomial.
No general principle
can apply to all finitely based classes, since
Garrabrant and Pak~\cite{GarrabrantPak2016} constructed one whose counting
sequence cannot be computed in polynomial time
unless $\mathsf{EXP}=\mathsf{\oplus EXP}$.\footnote{Here
$\mathsf{EXP}$ is the class of decision problems solvable in deterministic
exponential time, $\mathsf{\oplus EXP}$ is the class of problems that ask
whether a nondeterministic exponential-time machine has an odd number of
accepting computations, and the two classes are widely believed to differ.}

If $\alpha$ and $\beta$ are permutations of $[r]$ and $[s]$, their
\emph{direct sum}
$\alpha\oplus\beta$ is the word obtained by writing $\alpha$ and then
writing $\beta$ with every value increased by $r$.  We will focus on the following direct-sum
family.  Put
\begin{equation}\label{eq:family}
 \iota_d=12\cdots d,
 \qquad
 \beta_d=\iota_d\oplus231
 =12\cdots d\,(d+2)(d+3)(d+1).
\end{equation}
Thus
\[
 \beta_1=1342,\qquad \beta_2=12453,\qquad
 \beta_3=123564,\qquad \beta_4=1234675.
\]
The case $d=1$ is already enumerated algebraically by B\'ona
\cite{Bona1997}.
The family itself also appears in B\'ona~\cite{Bona2005}: from
$\beta_d=1\oplus\beta_{d-1}$ for $d\geq2$ he obtained a recursive deletion
rule for \emph{left-to-right minima} (entries smaller than every entry to
their left), and, starting from the Stanley--Wilf limit $8$ of
$1342$~\cite{Bona1997}, used it to determine the Stanley--Wilf limits of the
family \cite[Proposition~5.1 and Theorem~5.2]{Bona2005}.  The rule determines
the growth rates but not the exact counts.
The next case is $d=2$, giving $\beta_2 = 12453$.  The counting sequence begins
\begin{equation}\label{eq:first-terms}
  1,1,2,6,24,119,694,4581,33286,260927,2174398,19053058,\ldots.
\end{equation}

In the census of Clisby, Conway, Guttmann, and
Inoue~\cite{ClisbyConwayGuttmannInoue2022}, the $12453$ sequence was known
through $n=38$ from a specialized program of
Biers-Ariel~\cite{BiersArielA116485}, posted with OEIS
A116485~\cite{OEISA116485}.  The program uses a
composition-valued state of unbounded length, and the number of its states
grows exponentially.  Our starting point is the same recurrence: for $d=2$
the literal recurrence developed below coincides with Biers-Ariel's, state
for state (\cref{sec:literal}), and we extend it to the whole family
$\beta_d$.  We then show how
to shrink its exponential state space to polynomial size.

The family also has a description by a sorting device.  Albert, Homberger,
Pantone, Shar, and Vatter \cite{AlbertHombergerPantoneSharVatter2018} study
the $\mathcal C$-machine, which reads the entries $1,2,\ldots,n$ in order
and either outputs each one directly or pushes it anywhere into a container
whose contents must stay order-isomorphic to a permutation in $\mathcal C$,
and which outputs entries from the container only from its left end.  For
a permutation $\tau$ of $[k]$, the \emph{skew sum} $1\ominus\tau$ is the
permutation $k+1,\tau_1,\ldots,\tau_k$, which places a new largest entry
before $\tau$, and the \emph{complement} $\tau^{c}$ replaces each entry $x$
of $\tau$ by $k+1-x$.  Their Theorem~1.1 shows that the $\Av(B)$-machine
generates exactly $\Av(1\ominus B)$, where
$1\ominus B=\{1\ominus\tau:\tau\in B\}$, and they observe that $\Av(4213)$,
the complement of $\Av(1342)$, is generated by the $\Av(213)$-machine,
answering a question of B\'ona.  Since
$\beta_d^{\,c}=1\ominus\beta_{d-1}^{\,c}$ for $d\geq2$, the class
$\Av(\beta_d)$ is, up to complementation, the class generated by the
$\Av(\beta_{d-1}^{\,c})$-machine.  For $12453$ the container
class is $\Av(4213)$.  They prove that the generated class has a rational generating function when
$\mathcal C$ is finite, an algebraic one when $\mathcal C$ is bounded, that is,
when $|\mathcal C_n|$ is bounded by a constant, and a polynomial-time counting
algorithm when $|\mathcal C_n|$ is bounded by a polynomial in $n$.  The container
class $\Av(4213)$ has growth rate $8$, so none of these conditions applies to
$12453$.

Suppose $\pi$ avoids $\beta_d$.  Every occurrence of $12\cdots d$ in a
prefix of $\pi$ forces the unread values above it to be read in a
$231$-avoiding order, and we record the conditions created by a prefix as
an ordered stack of intervals of unread values (\cref{sec:obligations}).
Reading one more letter changes the control, that is, the counts of unread
values that no condition restricts, or the active head of the stack, or
both.  The deferred intervals are not touched until the head has been
consumed, and the transitions up to that point do not depend on them.  A
transfer kernel records the control at which this happens.  This removes
the stack from the memoization key, and the resulting table has polynomial
size.

Our main result is the following.
The notation $O_d(\cdot)$ allows the hidden constant to depend on the fixed
family parameter $d$.

\Needspace{12\baselineskip}
\begin{theorem}\label{thm:main}
For every fixed $d\geq1$ and every $N\geq1$, the numbers
\[
  \lvert\Av_0(\beta_d)\rvert,\ldots,\lvert\Av_N(\beta_d)\rvert
\]
can be computed using $O_d(N^{3d+2})$ integer arithmetic operations
and $O_d(N^{2d+1})$ stored integers.  All integers involved have
$O_d(N\log N)$ bits.
Therefore the bit complexity\footnote{The operation count treats each
integer addition or multiplication as one step.  The bit complexity counts
operations on single bits, so the sizes of the integers enter through the multiplication cost $M(b)$
defined below.}
is polynomial.  More precisely it is
\[
  O_d\!\left(N^{3d+2} M(N\log N)\right),
\]
where $M(b)$ is the cost of multiplying $b$-bit integers, and the bit
storage is $O_d(N^{2d+2}\log N)$.
\end{theorem}

In particular, the family theorem gives $O(N^5)$ operations and $O(N^3)$
stored integers for $\Av(1342)$, and $O(N^{8})$ and $O(N^5)$ for
$\Av(12453)$.  An additional translation symmetry, proved in
\cref{sec:d2-support}, improves this to $O(N^7)$ and $O(N^4)$ for
$\Av(12453)$ (\cref{cor:d2-complexity}).  The same symmetry holds for
every $d$ and lowers both exponents of \cref{thm:main} by one.  The
theorem gives a polynomial bound for each fixed $d$, not a uniform one when
$d$ is part of the input.

To our knowledge, no polynomial-time enumeration algorithm was
previously available for the family $\Av(\iota_d\oplus231)$ at every fixed
$d$, or even for the single class $\Av(12453)$.  In this paper we show that
the literal recurrence is a pushdown system: every transition acts on the
interval stack through its active head alone, so the completion counts
factor through transfer kernels
(\cref{thm:protected-tail-principle,cor:protected-tail}), although the
number of reachable stack states is at least the Fibonacci number
$F_{n-d}$ (\cref{prop:exponential}).  The factorization itself is the
counting form of the classical conversion of a pushdown automaton into a
grammar (\cref{rem:pop-sequences}).  What is new is the combinatorial input
that makes it apply: for every $d$, the conditions created by a prefix form a
stack that the scan touches only at its head
(\cref{lem:separators,prop:scan-states,prop:state-invariant}), and the number
of controls is polynomial in $n$, so the factorization gives a transfer table
of polynomial size rather than a finite grammar.  We prove the family bounds of
\cref{thm:main} and two translation quotients for $12453$
(\cref{lem:d2-translation,lem:d2-second-translation}).  We also provide
the values $\lvert\Av_n(12453)\rvert$ for $n\leq150$, computed by Chinese remaindering
with a proved bound and checked by separately written implementations
(\cref{tab:terms-150} and the supplement).  The
transfer tables can also be used to generate uniformly random elements of
$\Av_n(12453)$ in polynomial time (\cref{prop:sampling}), and the literal and
kernel recurrences for $d\leq2$ are machine-checked in Lean~4
(\cref{sec:formal}).

The paper is organized as follows.  In \cref{sec:obligations} we describe the
conditions that a scanned prefix imposes on its completions, represent them
by an interval stack, and develop the whole method for $d=1$, where
$\Av(1342)$ serves as a prototype: the literal recurrence, the protected-tail
factorization of \cref{thm:protected-tail-principle}, and the resulting
polynomial algorithm.  In \cref{sec:frontier} we replace the single minimum
by the $d$ patience-sorting thresholds and their bands.  In
\cref{sec:literal} we derive the literal recurrence for every $d$ and show
that its state space is exponential.  In \cref{sec:kernels} we apply the
factorization to it and obtain the transfer kernels.  In \cref{sec:algorithm}
we bound their supports and prove \cref{thm:main}.  In \cref{sec:d2-support}
we specialize to $12453$ and use a translation symmetry to reach $O(N^7)$
operations and $O(N^4)$ storage, and a second one to halve the storage
constant.  In \cref{sec:computations} we
compute the values $\lvert\Av_n(12453)\rvert$ for $n\leq150$ and describe the uniform
sampler.  In \cref{sec:formal} we describe the Lean development, with
details in \cref{app:lean}.  In \cref{sec:discussion} we place the class among the
length-five Wilf classes and state a conjecture on its asymptotics.
\Cref{app:float} bounds the distance of the floating-point sampler from
uniformity.

\section{Triggers, stacks, and the complete one-threshold prototype}
\label{sec:obligations}

We first describe the restrictions that a scanned prefix imposes on its
completions and show that they admit an interval-stack representation for
every $d\geq1$.  We then develop the counting algorithm completely for
$d=1$, which isolates the protected-tail construction before the additional
threshold coordinates are introduced.  The scan states and their legal
moves (\cref{def:scan-state,lem:legal-moves,lem:separators,prop:scan-states})
are defined and studied for every $d$ and used unchanged in
\cref{sec:literal}, and the proof of \cref{thm:protected-tail-principle}
applies verbatim in \cref{sec:kernels}.

Containment and avoidance are defined for words with distinct entries in
the same way, by relative order.  The \emph{standardization} of such a word is
obtained by replacing the smallest entry by $1$, the next smallest by $2$,
and so on.  If $X$ is a set of values, $w|_X$ denotes the
subword formed by retaining only the entries belonging to $X$.  An
\emph{increasing $d$-subsequence} of a permutation $\pi$, or more generally of a
word $\pi$ with distinct entries, is a choice of positions
$i_1<\cdots<i_d$ with
$\pi_{i_1}<\cdots<\pi_{i_d}$.  We call its last entry a
\emph{$d$-trigger}.  The positions need not be adjacent.

\begin{lemma}[Trigger lemma]\label{lem:trigger}
A permutation $\pi=\pi_1\cdots\pi_n$ avoids $\beta_d$ if and only if, for
every increasing $d$-subsequence ending at $c=\pi_j$, the word
\[
  \pi_{j+1}\cdots\pi_n\big|_{\{x:x>\pi_j\}}
\]
avoids $231$.
\end{lemma}

\begin{proof}
Suppose $\pi$ contains $\beta_d$.  The first $d$ entries of an occurrence
form an increasing $d$-subsequence, and its last three entries form a $231$ whose values are all
larger.  Conversely, an increasing $d$-subsequence followed by such a $231$
has relative order $\iota_d\oplus231=\beta_d$.
\end{proof}

When a trigger $c$ is read, its \emph{projection}, the subword of the
letters after $c$ that are larger than $c$, must avoid $231$.  Later in the scan, some letters of that projection have already been
read.  It is not enough to test the unread part of the projection, since an
occurrence of $231$ may use both read and unread letters.  The interval stack
below records the resulting conditions on the unread letters.

More formally, the \emph{$231$-obligation created by a trigger $c=\pi_j$}
requires that
\(
 \pi_{j+1}\cdots\pi_n|_{\{x:x>c\}}
\)
avoid $231$.  Let $\sigma=\sigma_1\cdots\sigma_k$ be
the scanned prefix.  The letters of $[n]$ not in $\sigma$ are \emph{unread},
and a \emph{completion} of $\sigma$ is a word $w$ using each unread letter
once.

\subsection{The Catalan interval stack}\label{sec:stack}

Here and throughout, an \emph{interval} is a consecutive block of the
ordered set of currently unread values.  Its elements need not be
consecutive integers.  The \emph{local rank} of an element $x$ of a finite
totally ordered set $I$ is $|\{y\in I:y\leq x\}|$.  The following
first-letter decomposition of $\Av(231)$ applies to an arbitrary finite
ordered alphabet.

\Needspace{12\baselineskip}
\begin{lemma}[First-letter lemma]\label{lem:first-letter}
Let $I$ be a finite totally ordered set of size $\ell$, and let $x\in I$ have
local rank $r$.  A permutation of $I$ beginning with
$x$ avoids $231$ if and only if
all values below $x$ occur before all values above $x$, and the two induced
subwords avoid $231$.  Thus reading $x$ replaces an interval of size $\ell$
by intervals of sizes $r-1$ and $\ell-r$, in that order, with zero parts
deleted.
\end{lemma}

\begin{proof}
Suppose an upper value occurs before a lower value after $x$.  Then $x$ and
these two values form a $231$.  Conversely, if the lower and upper subwords occur in
that order, a $231$ cannot use letters from both: both position and value
increase from the first block to the second.  Nor can a $231$ use $x$ and
two letters from just one block: $x$ is larger than every lower letter and
smaller than every upper letter, while the first entry of a $231$ has
middle rank.  If it uses $x$, one lower letter, and one upper letter, their
order is $213$.  Avoidance therefore reduces to avoidance inside the two
blocks.
\end{proof}

We write $\Av(231)(I)$ for the $231$-avoiding words using each value of $I$
once.  If $I<J$, that is, every element of $I$ is smaller than every
element of $J$, and $u$ and $w$ are words on $I$ and $J$, then $u\oplus w$
is the concatenation $uw$.  For sets of such words, $\oplus$ denotes the set
of all concatenations.  Consider the projection created by one trigger $c$,
and read its letters one at a time.  By \cref{lem:first-letter}, the first
letter $x$ splits the rest of the projection into a lower and an upper
interval, and every letter of the lower interval must be read before any
letter of the upper one.  The next letter of the projection therefore lies
in the lower interval while that interval is nonempty, and the lemma,
applied to the subword on that interval, splits it in turn.  The letter itself then
plays no further role, since both conditions of the lemma concern only the
letters after it.  Repeated application
of \cref{lem:first-letter} thus represents what the obligation of $c$
requires of the unread letters by an ordered interval stack
\[
  I_1<I_2<\cdots<I_s,
\]
which we also write $I_1\mid I_2\mid\cdots\mid I_s$.  That is, the unread letters of the projection must form a word in
\begin{equation}\label{eq:stack-form}
  \Av(231)(I_1)\oplus\Av(231)(I_2)\oplus\cdots
  \oplus\Av(231)(I_s).
\end{equation}
The read letters of the projection do not appear in \eqref{eq:stack-form},
and none of them is needed again.  For example, relabel the values of a projection as
$1,\ldots,6$ and suppose that its first two letters are $4$ and $2$.
Reading $4$ leaves the stack $\{1,2,3\}\mid\{5,6\}$, and reading $2$ then
leaves $\{1\}\mid\{3\}\mid\{5,6\}$.  The completion $1365$ is allowed,
while a completion beginning $31$ is excluded: together with the read
letter $2$ it would form a $231$.

The stack records the obligations created by the scanned prefix.  Further
obligations arise as the completion is read, and they are incorporated by
the stack updates described below.  We call $I_1$ the \emph{active head} and $(I_2,\ldots,I_s)$ the
\emph{deferred tail}.  Only the active head may be read next among the values
of the stack: every letter of $I_1$ must be read before any letter of
$I_2$, and so forth.  Unread values below the trigger are constrained by no
obligation created so far and may be interspersed.  For an interval of size $\ell$, an
\emph{endpoint choice} has local rank $1$ or $\ell$, and an \emph{interior
choice} has local rank in $\{2,\ldots,\ell-1\}$.

We call the gap in value
between two consecutive intervals of the current stack a \emph{separation},
and a read letter lying in that gap a \emph{separator}.  A separation may
have several separators.  An interior choice creates a new separation, at
the letter read, while an endpoint choice leaves at least one side empty
and creates none, although the letter read may lie in an existing
separation.  \Cref{lem:separators} below determines the separations of the
stack once the obligations of several triggers have been combined.

\begin{example}[Reading one projection]\label{ex:one-projection}
The permutation
\begin{equation}\label{eq:running-example}
 \pi=9,11,10,14,5,12,6,2,3,8,4,7,1,13,15
\end{equation}
of $[15]$ avoids $1342$, and therefore also $12453$, and serves as the
running example of this paper.  For $d=1$ every letter is a trigger.  The
projection created by the first letter $9$ is the word
$11,10,14,12,13,15$ formed by the later letters above $9$.  It
standardizes to $215346$ and avoids $231$, as \cref{lem:trigger} requires.
We read this projection letter by letter.  The stack starts as the single
interval $\{10,\ldots,15\}$.  The first letter $11$ has local rank $2$, an
interior choice, so \cref{lem:first-letter} requires the lower interval
$\{10\}$ to be read before the upper interval $\{12,\ldots,15\}$.  The stack
is $\{10\}\mid\{12,\ldots,15\}$.  Reading $10$ is an endpoint choice that
exhausts the head and leaves $\{12,\ldots,15\}$, and the interior choice
$14$, of local rank $3$ in that interval, splits it into
$\{12,13\}\mid\{15\}$.  The unread letters of the projection must now form
a word in $\Av(231)(\{12,13\})\oplus\Av(231)(\{15\})$: either order of $12$
and $13$, then $15$.  The read letters $11$, $10$, $14$ are not retained.
While $10$
was unread, $11$ was the only separator of the separation between $\{10\}$
and $\{12,\ldots,15\}$.  Now $14$ is the only separator of the separation
between $\{12,13\}$ and $\{15\}$, and $10$ and $11$ separate nothing.
\end{example}

Several triggers impose overlapping obligations.  We now define the object
that records all of them: a read prefix together with a stack, updated by
three kinds of move.  \Cref{prop:scan-states} then shows that the legal
prefixes are exactly the prefixes of $\beta_d$-avoiding permutations.

\begin{definition}[Scan states and legal moves]\label{def:scan-state}
Fix $d\geq1$ and $n$.  A \emph{scan state} is a pair $(\sigma,S)$.  Here
$\sigma$ is a word with distinct letters from $[n]$, called the \emph{read
prefix}, and the letters of $[n]$ not in $\sigma$ are \emph{unread}.  We
write $q$ for the least $d$-trigger of $\sigma$, and put $q=n+d$, a virtual
value above every letter, when $\sigma$ has no $d$-trigger.  The
\emph{stack} $S=(I_1,\ldots,I_s)$ is a list of nonempty sets with
$I_1<I_2<\cdots<I_s$ whose union is the set of unread letters larger than
$q$, so that each $I_j$ is an interval in the sense of \cref{sec:stack},
$I_1$ is the active head, and $(I_2,\ldots,I_s)$ is the deferred tail.  In
particular $S=\varnothing$ is the empty list, $s=0$, when $\sigma$ has
no $d$-trigger.  The \emph{initial state} is $(\emptyword,\varnothing)$: the
empty prefix and the empty stack.  Let $(\sigma,S)$ be a scan state and let $x$ be an unread letter.
Reading $x$ is a \emph{legal move}, leading to the state $(\sigma x,S')$, in
the following three cases.
\begin{enumerate}
\item[(a)] \emph{Merger}: $x$ is a $d$-trigger of $\sigma x$ and $x<q$.  Let
  $E$ be the set of unread letters strictly between $x$ and $q$.  Then
  $S'=(E\cup I_1,I_2,\ldots,I_s)$, where $I_1=\varnothing$ if $s=0$, with an
  empty set deleted.
\item[(b)] \emph{Split}: $x$ is a $d$-trigger of $\sigma x$, $x>q$, and $x$
  lies in the active head $I_1$.  Then $S'=(I_1^{<x},I_1^{>x},I_2,\ldots,I_s)$, where $I_1^{<x}$
  and $I_1^{>x}$ are the sets of elements of $I_1$ below and above $x$, with
  empty sets deleted.  The move is an endpoint or an interior choice, as defined above,
  depending on the local rank of $x$ in $I_1$.
\item[(c)] \emph{Non-trigger move}: $x$ is not a $d$-trigger of $\sigma x$.
  Then $S'=S$.  For $d=1$ every letter is a $1$-trigger, so this case does
  not occur.
\end{enumerate}
Reading any other unread letter is \emph{illegal}.  A word $\sigma$ is
\emph{legal} if the initial state reaches a state $(\sigma,S)$ by legal
moves, and we write $\mathcal S(\sigma)$ for the resulting stack, the
\emph{stack of $\sigma$}.
\end{definition}

A $d$-trigger read from a prefix without one always merges, since then
$q=n+d$.  Each move determines $S'$ from $(\sigma,S)$ and
$x$, and \cref{lem:legal-moves}(a) below shows that $(\sigma x,S')$ is again
a scan state, so the stack $\mathcal S(\sigma)$ of a legal word is
determined by $\sigma$.  Note that every unread letter $x>q$ is a
$d$-trigger of $\sigma x$, since
replacing the last entry of an increasing $d$-subsequence ending at $q$ by
$x$ gives one ending at $x$.  For $d=1$ a legal move is a new left-to-right
minimum, which merges (\cref{fig:moves}(a)), or a letter of the active head,
which splits (\cref{fig:moves}(b)).  \Cref{fig:merger} shows the merger of
the running example in detail.

\begin{figure}[b]
\centering
\begin{tikzpicture}[x=1cm,y=1cm,>=stealth,font=\footnotesize]
  \draw[->] (-0.5,-1.1) -- (-0.5,2.9);
  \node[rotate=90,anchor=south] at (-0.7,0.6) {value};
  \fill[tailorange!25] (0,2.0) rectangle (1.5,2.8);
  \draw (0,2.0) rectangle (1.5,2.8);
  \node at (0.75,2.4) {$I_2,\ldots,I_s$};
  \fill[pairblue!25] (0,0.6) rectangle (1.5,1.8);
  \draw (0,0.6) rectangle (1.5,1.8);
  \node at (0.75,1.2) {$I_1$};
  \fill[darkgreen!25] (0,-0.8) rectangle (1.5,0.4);
  \draw[dashed] (0,-0.8) rectangle (1.5,0.4);
  \node at (0.75,-0.2) {$E$};
  \fill[red] (0.75,-1.3) circle (2.2pt);
  \node[red,right=1pt] at (0.85,-1.3) {$x$};
  \draw[->] (1.75,0.8) -- (2.15,0.8);
  \fill[tailorange!25] (2.4,2.0) rectangle (3.9,2.8);
  \draw (2.4,2.0) rectangle (3.9,2.8);
  \node at (3.15,2.4) {$I_2,\ldots,I_s$};
  \fill[pairblue!25] (2.4,-0.8) rectangle (3.9,1.8);
  \draw (2.4,-0.8) rectangle (3.9,1.8);
  \node at (3.15,0.5) {$E\cup I_1$};
  \node at (1.95,-2.0) {(a) merger};
  \fill[tailorange!25] (4.8,2.0) rectangle (6.3,2.8);
  \draw (4.8,2.0) rectangle (6.3,2.8);
  \node at (5.55,2.4) {$I_2,\ldots,I_s$};
  \fill[pairblue!25] (4.8,0.6) rectangle (6.3,1.8);
  \draw (4.8,0.6) rectangle (6.3,1.8);
  \node at (5.55,1.55) {$I_1$};
  \fill[red] (5.4,1.0) circle (2.2pt);
  \node[red,right=1pt] at (5.5,1.0) {$x$};
  \draw[->] (6.55,0.8) -- (6.95,0.8);
  \fill[tailorange!25] (7.2,2.0) rectangle (8.7,2.8);
  \draw (7.2,2.0) rectangle (8.7,2.8);
  \node at (7.95,2.4) {$I_2,\ldots,I_s$};
  \fill[pairblue!25] (7.2,1.25) rectangle (8.7,1.8);
  \draw (7.2,1.25) rectangle (8.7,1.8);
  \node at (7.95,1.525) {$I_1^{>x}$};
  \fill[pairblue!25] (7.2,0.6) rectangle (8.7,1.15);
  \draw (7.2,0.6) rectangle (8.7,1.15);
  \node at (7.95,0.875) {$I_1^{<x}$};
  \node at (6.75,-2.0) {(b) split};
  \fill[tailorange!25] (9.6,2.0) rectangle (11.1,2.8);
  \draw (9.6,2.0) rectangle (11.1,2.8);
  \node at (10.35,2.4) {$I_2,\ldots,I_s$};
  \fill[pairblue!25] (9.6,0.6) rectangle (11.1,1.8);
  \draw (9.6,0.6) rectangle (11.1,1.8);
  \node at (10.35,1.2) {$I_1$};
  \fill[red] (10.35,-1.3) circle (2.2pt);
  \node[red,right=1pt] at (10.45,-1.3) {$x$};
  \draw[->] (11.35,0.8) -- (11.75,0.8);
  \fill[tailorange!25] (12.0,2.0) rectangle (13.5,2.8);
  \draw (12.0,2.0) rectangle (13.5,2.8);
  \node at (12.75,2.4) {$I_2,\ldots,I_s$};
  \fill[pairblue!25] (12.0,0.6) rectangle (13.5,1.8);
  \draw (12.0,0.6) rectangle (13.5,1.8);
  \node at (12.75,1.2) {$I_1$};
  \node at (11.55,-2.0) {(c) non-trigger move};
\end{tikzpicture}
\caption{The three legal moves of \cref{def:scan-state}, with values
increasing upwards and the active head $I_1$ at the bottom of the stack.
In a merger the letter $x$ lies below the least $d$-trigger $q$, and the
unread values $E$ strictly between $x$ and $q$ join the head.  In a split
$x$ lies in the head and cuts it into the values below and above $x$.  A
non-trigger move leaves the stack unchanged.  The deferred tail is unchanged
in all three moves.}
\label{fig:moves}
\end{figure}
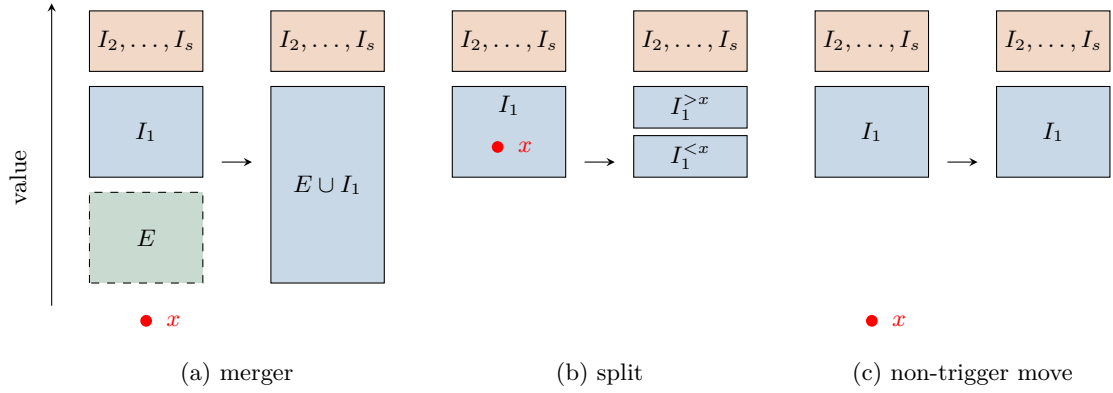

\begin{lemma}[Legal moves]\label{lem:legal-moves}
Let $(\sigma,S)$ be a scan state and let $x$ be an unread letter.
\begin{enumerate}
\item[(a)] If reading $x$ is a legal move, then $(\sigma x,S')$, with $S'$
  as in \cref{def:scan-state}, is a scan state.
\item[(b)] Reading $x$ is illegal if and only if $x$ lies in one of the deferred
  intervals $I_2,\ldots,I_s$.
\end{enumerate}
\end{lemma}

\begin{proof}
(a) In a merger, $x$ becomes the least $d$-trigger of $\sigma x$, and the
unread letters of $\sigma x$ larger than $x$ are those of $E$ together with
those larger than $q$, that is, $E\cup I_1\cup\cdots\cup I_s$, the union of
$S'$.  Every letter of $E$ is smaller than $q$ and every letter of $S$ is
larger, so $E<I_1$ and the sets of $S'$ are disjoint and increasing.  In a
split, the least $d$-trigger is still $q$, since $x>q$, and the unread
letters above $q$ lose exactly $x$, which is what the split removes.  The two
parts of $I_1$ lie below and above $x$.  In a non-trigger move, $x<q$, since
an unread $x>q$ is a $d$-trigger of $\sigma x$ as noted above, so $q$ and the
unread letters above it are unchanged, and so is $S$.

(b) Let $x$ lie in a deferred interval.  Then $x>q$, so $x$ is a
$d$-trigger of $\sigma x$, as noted above.  Therefore cases (a)
and (c) of \cref{def:scan-state} do not apply, and $x\notin I_1$ excludes
case (b).  Conversely, let reading $x$ be illegal.  Since case (c) does not
apply, $x$ is a $d$-trigger of $\sigma x$, and since case (a) does not apply,
$x>q$.  Therefore $x$ lies in the
stack, and in a deferred interval since case (b) does not apply.
\end{proof}

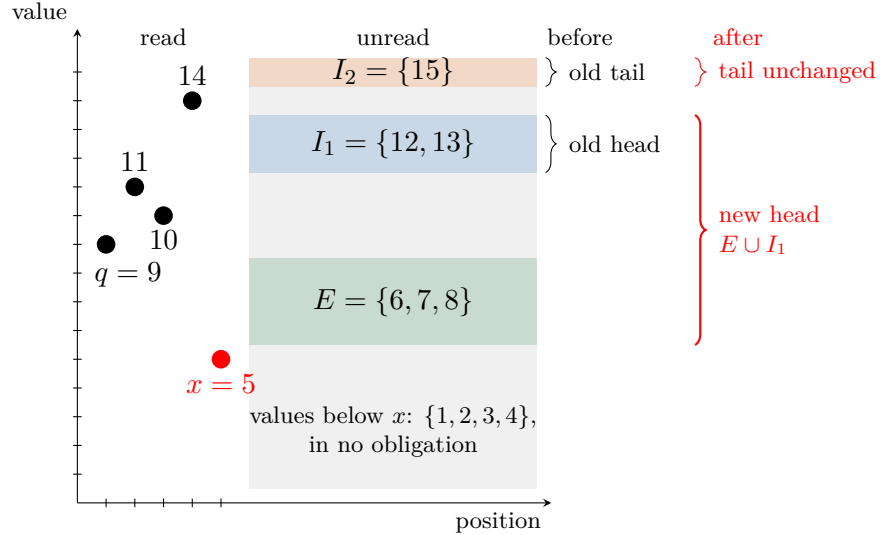
\begin{figure}[b]
\centering
\begin{tikzpicture}[x=0.38cm,y=0.38cm,>=stealth]
  \draw[->] (0,0) -- (16.5,0) node[below left] {\footnotesize position};
  \draw[->] (0,0) -- (0,16.5) node[above left] {\footnotesize value};
  \fill[black!6] (6,0.5) rectangle (16,15.5);
  \node[anchor=south] at (11,15.6) {\footnotesize unread};
  \node[anchor=south] at (3,15.6) {\footnotesize read};
  \fill[tailorange!25] (6,14.5) rectangle (16,15.5);
  \fill[pairblue!25]   (6,11.5) rectangle (16,13.5);
  \fill[darkgreen!25]  (6,5.5)  rectangle (16,8.5);
  \node at (11,15)   {$I_2=\{15\}$};
  \node at (11,12.5) {$I_1=\{12,13\}$};
  \node at (11,7)    {$E=\{6,7,8\}$};
  \node[align=center] at (11,2.5) {\footnotesize values below $x$: $\{1,2,3,4\}$,\\[-2pt]\footnotesize in no obligation};
  \fill (1,9) circle (3.5pt);  \node[below left=1pt,anchor=north west] at (0.3,8.8) {$q=9$};
  \fill (2,11) circle (3.5pt); \node[above=2pt] at (2,11) {$11$};
  \fill (3,10) circle (3.5pt); \node[below=2pt] at (3,10) {$10$};
  \fill (4,14) circle (3.5pt); \node[above=2pt] at (4,14) {$14$};
  \fill[red] (5,5) circle (3.5pt); \node[below=2pt,red] at (5,5) {$x=5$};
  \foreach \y in {1,...,15} \draw (-0.15,\y) -- (0.15,\y);
  \foreach \x in {1,...,5} \draw (\x,-0.15) -- (\x,0.15);
  \draw[decorate,decoration={brace,mirror,amplitude=4pt}] (16.3,14.5) -- (16.3,15.5)
    node[midway,right=5pt] {\footnotesize old tail};
  \draw[decorate,decoration={brace,mirror,amplitude=4pt}] (16.3,11.5) -- (16.3,13.5)
    node[midway,right=5pt] {\footnotesize old head};
  \draw[decorate,decoration={brace,mirror,amplitude=4pt},red] (21.5,14.5) -- (21.5,15.5)
    node[midway,right=5pt,red] {\footnotesize tail unchanged};
  \draw[decorate,decoration={brace,mirror,amplitude=4pt},red,thick] (21.5,5.5) -- (21.5,13.5)
    node[midway,right=5pt,red,align=left] {\footnotesize new head\\[-2pt]\footnotesize $E\cup I_1$};
  \node[anchor=south] at (17.5,15.6) {\footnotesize before};
  \node[anchor=south,red] at (23,15.6) {\footnotesize after};
\end{tikzpicture}
\caption{A merger (\cref{def:scan-state}(a)) for $d=1$ on the running
example \eqref{eq:running-example}: the prefix $9,11,10,14$ has been read,
leaving the head $I_1=\{12,13\}$ and the tail $I_2=\{15\}$, and the new
trigger $x=5$ merges $E=\{6,7,8\}$, the unread values between $x$ and
$q=9$, into the head.  The tail is unchanged.  See \cref{ex:one-merger}.}
\label{fig:merger}
\end{figure}

\begin{example}[One merger]\label{ex:one-merger}
We continue \cref{ex:one-projection}.  The fifth letter $5$ of $\pi$ is a
trigger smaller than the earlier triggers $9,11,10,14$, so reading it is a
merger, and \cref{fig:merger} shows the stack before and after.  In the
notation of \cref{def:scan-state}(a), $q=9$ and $E=\{6,7,8\}$, the unread
values between $5$ and $9$.  The stack $\{12,13\}\mid\{15\}$ consists of the
unread values above $9$, and the merged stack is
$\{6,7,8,12,13\}\mid\{15\}$.
The values $6,7,8$ join the head without a separation: the read letters
$9$, $10$, $11$, which lie between them and $12$ in value, were read before
the trigger $5$ and constrain nothing in $E$.  They must, however, precede
$15$: if $15$ came before some $y'\in\{6,7,8\}$, then either $15$
precedes some $y\in\{12,13\}$ and $9,14,15,y$ is an occurrence of $1342$,
or $12$ precedes $15$ and $5,12,15,y'$ is one.  Inside the head only $231$-avoidance
remains.  The permutation $\pi$ continues with $12$ before $6$, and
$13,12,6,7,8$ would have been allowed as well.  For $d=2$ the same permutation is scanned in
\cref{ex:full-state}: there the first $2$-trigger is $11$ rather than $9$,
and the corresponding merger occurs at the $2$-trigger $6$, with $q=10$ and
$E=\{7,8\}$.
\end{example}

Two unread letters are \emph{adjacent} if no unread letter lies strictly
between them.  The following lemma determines the separations of the stack
from the scanned prefix.

\begin{lemma}[Separations of the stack]\label{lem:separators}
Let $\sigma$ be a legal word, let $q$ be its least $d$-trigger, and let
$u<v$ be adjacent unread letters above $q$.  Then $u$
and $v$ lie in different intervals of $\mathcal S(\sigma)$ if and only if
some letter $x$ with $u<x<v$ was read after a $d$-trigger smaller than $u$.
\end{lemma}

\begin{proof}
We argue by induction along the legal moves that read $\sigma$.  While no
$d$-trigger has been read there is no letter above $q=n+d$.  Let $x$ be the letter
read next.  Reading $x$ makes its two unread neighbours adjacent, when both
exist, and leaves every other adjacent pair with its members and the letters
between them.

Suppose first that reading $x$ is a merger.  Then $x$ becomes the least
$d$-trigger, and an adjacent pair $u<v$ above $x$ that was not an adjacent
pair above $q$ has $u\in E$.  Every $d$-trigger of $\sigma$ is larger than
$u$, and no letter has been read after $x$, so the second condition fails.
The stack of $\sigma$ consists of the unread letters above $q$, so $v$ lies
in $E$ or is the least element of $I_1$, and $u$ and $v$ lie in the single
interval $E\cup I_1$.  Thus the first condition fails as well.  The
other pairs keep their intervals, and $x$ lies below them and has not been
followed by any letter, so both conditions are unchanged for them.

Suppose next that reading $x$ is a split.  Then $x\in I_1$ and $x>q$.  A new
adjacent pair consists of the unread neighbours $u<x<v$ of $x$, when both
exist and $u>q$.  Then $u$ lies in the part of $I_1$ below $x$, while $v$
lies in the part above $x$ or, if $x=\max I_1$, is the least element of
$I_2$.  Thus $u$ and $v$ lie in different intervals, and $x$, read after
$q<u$, witnesses the second condition.  The other pairs keep their
intervals, and $x$ lies between the members of none of them, since those are
adjacent and $x$ was unread.

Finally, a non-trigger move reads a letter $x<q$, as in the proof of
\cref{lem:legal-moves}, and changes neither the stack nor either condition.
\end{proof}

\begin{proposition}[Scan states of avoiders]\label{prop:scan-states}
Let $\sigma$ be a legal word with stack $\mathcal S(\sigma)=(I_1,\ldots,I_s)$.
\begin{enumerate}
\item[(a)] The word $\sigma$ avoids $\beta_d$.
\item[(b)] If reading the unread letter $x$ from $\sigma$ is illegal, then no
  $\beta_d$-avoiding permutation has the prefix $\sigma x$.
\item[(c)] A word is legal if and only if it is a prefix of a
  $\beta_d$-avoiding permutation.  In particular $\Av_n(\beta_d)$ is the set
  of legal permutations of $[n]$.
\end{enumerate}
\end{proposition}

\begin{proof}
(a) Suppose that $\sigma$ contains $\beta_d$.  The first $d$ letters of an
occurrence form an increasing $d$-subsequence ending at a $d$-trigger $c$,
and its last three letters $y_2,y_3,y_1$, in this order, satisfy
$c<y_1<y_2<y_3$.  Let $\sigma'$ be the prefix of $\sigma$ read just
before $y_3$.  It is legal and contains $c$ and $y_2$, so its least
$d$-trigger $q$ satisfies $q\leq c$.  Since $y_3>q$ is read legally from
$\sigma'$, it lies in the active head of $\mathcal S(\sigma')$ by
\cref{lem:legal-moves}(b).  Let $u$ and $v$ be the unread letters of
$\sigma'$ next below and next above $y_2$.  Then $q<y_1\leq u$ and
$v\leq y_3$, and $y_2$ was read after the $d$-trigger $c<u$, so $u$ and $v$
lie in different intervals of $\mathcal S(\sigma')$ by \cref{lem:separators}.
But the stack consists of the unread letters above $q$, its active head is
the lowest interval, and $u<v\leq y_3$, so $u$ and $v$ both lie in the
active head, a contradiction.

(b) By \cref{lem:legal-moves}(b), $x$ lies in a deferred interval, so
$s\geq2$.  Let $u=\max I_1$ and $v=\min I_2$, which are adjacent.  By
\cref{lem:separators}, some letter $z$ with $u<z<v$ was read after a
$d$-trigger $c<u$.  A permutation with prefix $\sigma x$ reads $u$ after
$x$, and $v\leq x$.  An increasing $d$-subsequence ending at $c$, followed
by $z$, $x$, $u$, is therefore an occurrence of $\beta_d$, since
$c<u<z<x$.

(c) Every prefix of an avoider is legal: the empty prefix is, and if a
prefix $\sigma'$ of an avoider $\pi$ is legal and the next letter $x$ of
$\pi$ were illegal, (b) would forbid the prefix $\sigma'x$ of $\pi$.
Conversely, a legal $\sigma$ has an avoiding completion: read the unread
letters smaller than $q$ in decreasing order, and then the intervals
$I_1,\ldots,I_s$ in order, each in decreasing order.  Each letter of the
first group is the largest unread letter below $q$, so it is a non-trigger
move or a merger with $E=\varnothing$, and the stack is unchanged.  Each
later letter is the greatest element of the active head, and therefore a
$d$-trigger above $q$, so reading it is an endpoint choice.  All moves are
legal, so the completed permutation is legal and avoids $\beta_d$ by (a).
The last sentence of (c) follows, with (a) for the legal permutations.
\end{proof}

Being a prefix of a $\beta_d$-avoiding permutation is stronger than
avoiding $\beta_d$.  The word $134$ avoids $1342$, but its only completion in $[4]$
is $1342$, and after $1,3$ the stack is $\{2\}\mid\{4\}$, so reading
$4$ is illegal.

\subsection{The one-threshold state}

For $d=1$, every letter is a trigger, so the least $d$-trigger $q$ of
\cref{def:scan-state} is the least letter $m$ read so far, virtual for the
empty prefix.  A new left-to-right minimum joins the unread values between
it and $m$ to the active head, and a letter of the active head shortens the
head at an endpoint choice and splits it at an interior choice.  By
\cref{lem:separators}, two adjacent unread letters $u<v$ above $m$ lie in
different intervals exactly when some letter between them was read at a time
when the minimum read so far was already below $u$, so the stack is
recovered from the scanned prefix alone.  Adjacency cannot be dropped:
after the prefix $3,5,1$ of a $1342$-avoiding permutation of $[7]$ the
stack is $\{2,4\}\mid\{6,7\}$, so $2$ and $6$ lie in different intervals,
although no letter between them was read after a letter smaller than $2$.

\begin{example}[A one-threshold scan]\label{ex:1342-scan}
We scan the whole permutation $\pi$ of \eqref{eq:running-example}, which can be
read down the second column of the table below.
The table lists, after each letter $\pi_i$, the type of move, the current
minimum $m$, the interval stack, and the symbol $W_p(L)$ of
\cref{sec:one-threshold}, where $p$ is the number of unread letters below
$m$, $L$ lists the interval sizes, and $W_p(L)$ is the number of
$1342$-avoiding completions of the prefix.  Every new minimum is a merger
in the sense of \cref{def:scan-state}(a).  The table writes ``new minimum, merger'' when
the new minimum joins values to the head or creates the stack and ``new
minimum'' alone when $E=\varnothing$.  It writes ``head exhausted'' for an
endpoint choice in a head of size one.
\begin{center}
\small
\setlength{\tabcolsep}{4.5pt}
\begin{tabular}{@{}rrlrll@{}}
\toprule
$i$ & $\pi_i$ & move & $m$ & interval stack & $W_p(L)$ \\
\midrule
$1$ & $9$ & new minimum, merger & $9$ & $\{10,\ldots,15\}$ & $W_{8}((6))$ \\
$2$ & $11$ & interior choice & $9$ & $\{10\}\mid\{12,\ldots,15\}$ & $W_{8}((1,4))$ \\
$3$ & $10$ & head exhausted & $9$ & $\{12,\ldots,15\}$ & $W_{8}((4))$ \\
$4$ & $14$ & interior choice & $9$ & $\{12,13\}\mid\{15\}$ & $W_{8}((2,1))$ \\
$5$ & $5$ & new minimum, merger & $5$ & $\{6,7,8,12,13\}\mid\{15\}$ & $W_{4}((5,1))$ \\
$6$ & $12$ & interior choice & $5$ & $\{6,7,8\}\mid\{13\}\mid\{15\}$ & $W_{4}((3,1,1))$ \\
$7$ & $6$ & endpoint choice & $5$ & $\{7,8\}\mid\{13\}\mid\{15\}$ & $W_{4}((2,1,1))$ \\
$8$ & $2$ & new minimum, merger & $2$ & $\{3,4,7,8\}\mid\{13\}\mid\{15\}$ & $W_{1}((4,1,1))$ \\
$9$ & $3$ & endpoint choice & $2$ & $\{4,7,8\}\mid\{13\}\mid\{15\}$ & $W_{1}((3,1,1))$ \\
$10$ & $8$ & endpoint choice & $2$ & $\{4,7\}\mid\{13\}\mid\{15\}$ & $W_{1}((2,1,1))$ \\
$11$ & $4$ & endpoint choice & $2$ & $\{7\}\mid\{13\}\mid\{15\}$ & $W_{1}((1,1,1))$ \\
$12$ & $7$ & head exhausted & $2$ & $\{13\}\mid\{15\}$ & $W_{1}((1,1))$ \\
$13$ & $1$ & new minimum & $1$ & $\{13\}\mid\{15\}$ & $W_{0}((1,1))$ \\
$14$ & $13$ & head exhausted & $1$ & $\{15\}$ & $W_{0}((1))$ \\
$15$ & $15$ & head exhausted & $1$ & $\varnothing$ & $W_{0}(\varnothing)$ \\
\bottomrule
\end{tabular}
\end{center}
The first five rows are \cref{ex:one-projection,ex:one-merger}.  In the
terms of \cref{lem:separators}, the letters $9$, $10$, and $11$
separate nothing at row $5$, as in \cref{ex:one-merger}: the separation
produced by $11$ at row $2$ disappeared once the interval below it was
exhausted, and the merger extended the head past it.  The letter $14$, read
after $9<13$, separates $13$ from $15$ from row $4$ on.  The interior choice $12$ then creates
three intervals, separated between $8$ and $13$ and between $13$ and $15$.
After the endpoint choice $6$, the new minimum $2$ joins $\{3,4\}$ to the head.  The
head $\{3,4,7,8\}$ contains the read letters $5$ and $6$ in its range,
again without a separation.  At row $8$ the unread letters above $m=2$ are
$3,4,7,8,13,15$, and the lemma cuts this list between $8$ and $13$, with
witness $12$ read after $5<8$, and between $13$ and $15$, with witness
$14$, and nowhere else.  This is the stack of the table.  The head is then
consumed in the $231$-avoiding order $3,8,4,7$.  The new minimum $1$ joins nothing to the head, and $13$, $15$ complete the
scan.  Throughout, the later minima $5$, $2$, and $1$ are read between
letters of the stack, and only sizes enter the state: the heads
$\{6,7,8,12,13\}$ of row $5$ and $\{3,4,7,8\}$ of row $8$ contribute to $L$
only their sizes $5$ and $4$.
\end{example}

\subsection{The one-threshold recurrence}\label{sec:one-threshold}

Recall that $m$ is the least letter in the scanned prefix, virtual for the
empty prefix.  Let $p$ be the number of unread letters below $m$, which we
call the \emph{control}, and let $L=(\ell_1,\ldots,\ell_s)$ be
the positive sizes of the active and deferred intervals.  We call $(p,L)$
the \emph{state} of the prefix.  We write
$|L|=\sum_{i=1}^s\ell_i$ and
$\operatorname{len}(L)=s$.  If $L'$ is a list, $(\ell)L'$ and
$(a,b)L'$ denote concatenation with the displayed initial entries, and
$\nz(x_1,\ldots,x_r)$ denotes the list obtained by deleting every zero
entry.  In the equations below, each summand stands for reading one letter
and moving from the current state to the state indexed by the summand.  We
call such a move a \emph{transition}.

\begin{proposition}[One-threshold recurrence]\label{prop:scalar-literal}
For $p\geq0$ and every list $L$ of positive integers, define $W_p(L)$ as
follows.  For $L=(\ell)L'$ with $\ell\geq1$,
\begin{align}
W_p((\ell)L')={}&\sum_{h=0}^{p-1}
 W_h\!\left((\ell+p-1-h)L'\right)
 +\sum_{\substack{a,b\geq0\\a+b=\ell-1}}W_p(\nz(a,b)L'), \label{eq:W}
\end{align}
and
\begin{equation}\label{eq:W-boundary}
 W_p(\varnothing)=
  \begin{cases}
   1,&p=0,\\
   W_{p-1}(\varnothing)+\sum_{h=0}^{p-2}W_h\!\left((p-1-h)\right),&p\geq1.
  \end{cases}
\end{equation}
For every $p$ and $L$, $W_p(L)$ is the number of maximal transition
sequences from $(p,L)$, all of which end at $(0,\varnothing)$.  For every
legal prefix with state $(p,L)$, reading the letters of a completion one at
a time is a bijection from the completions that produce a $1342$-avoiding
permutation to the maximal transition sequences from $(p,L)$.  In
particular
\begin{equation}\label{eq:W-initial}
 \lvert\Av_n(1342)\rvert=W_n(\varnothing).
\end{equation}
\end{proposition}

\begin{proof}
The quantity
\begin{equation}\label{eq:rho-scalar}
 \rho_1(p,L)=p+|L|
\end{equation}
decreases by one under every transition, and $(0,\varnothing)$ is the only
state without transitions.  Therefore the equations define $W_p(L)$, and by
induction on $\rho_1$ it is the number of maximal transition sequences from
$(p,L)$, all of which end at $(0,\varnothing)$.

By \cref{lem:legal-moves}(a), a legal prefix with state $(p,L)$ has
$\rho_1(p,L)$ unread letters, and we argue by induction on this number.  If
no letter is unread, the prefix is a permutation, it avoids $1342$ by
\cref{prop:scan-states}(a), and its state is $(0,\varnothing)$, from which the
only maximal sequence is empty.  Otherwise we partition the possible next
letters.  By \cref{prop:scan-states}(b), an illegal letter begins no
avoiding completion, and the legal letters are the new minima and the
letters of the active head.  If the next letter is below $m$, let $h$ be the
number of unread letters below it.  There is one letter for each
$0\leq h<p$.  The new minimum leaves $h$ unrestricted lower letters, while
the $p-1-h$ letters between it and $m$ join the active head.  This gives the
first sum in \eqref{eq:W}.  A letter of the active head of local rank $r$
replaces the head by its parts below and above it, of sizes $a=r-1$ and
$b=\ell-r$, so that $L$ becomes $\nz(a,b)L'$ by \cref{def:scan-state}(b).
This gives the last sum in \eqref{eq:W}, one letter for each pair $(a,b)$
with $a+b=\ell-1$.  With an empty
stack only a new minimum can be read, which gives \eqref{eq:W-boundary}.  In
every case the letter is recovered from its transition by its local rank,
and the induction hypothesis applies to the successor.  For the empty
prefix, $p=n$ and $L=\varnothing$, which is \eqref{eq:W-initial}.
\end{proof}

\begin{example}[All one-threshold moves]\label{ex:1342-state}
After the prefix $9,11,10,14,5,12$ of the running example
\eqref{eq:running-example}, row $6$ of the table in \cref{ex:1342-scan}, the
unread letters below the current minimum $5$ are $1,2,3,4$ and the stack is
$\{6,7,8\}\mid\{13\}\mid\{15\}$.  The state is $(p,L)=(4,(3,1,1))$, and
\eqref{eq:W} gives, with the two endpoint terms collected,
\begin{align*}
 W_4((3,1,1))={}&W_0((6,1,1))+W_1((5,1,1))+W_2((4,1,1))+W_3((3,1,1))\\
 &+2W_4((2,1,1))+W_4((1,1,1,1)).
\end{align*}
The first four terms correspond to reading $1$, $2$, $3$, or $4$, which
join $\{2,3,4\}$, $\{3,4\}$, $\{4\}$, or nothing to the active head.
The fifth term records the two endpoints $6,8$ of the active head, and
the last records its interior letter $7$.  Neither $13$ nor $15$ can be
read legally until the active head is empty.  The scan of
\cref{ex:1342-scan} continues with the endpoint $6$.
\end{example}

\begin{example}[Counting $\Av_5(1342)$ by the literal recurrence]
\label{ex:1342-literal-count}
By \eqref{eq:W-initial} and \eqref{eq:W-boundary},
\[
 \lvert\Av_5(1342)\rvert=W_5(\varnothing)
 =W_0((4))+W_1((3))+W_2((2))+W_3((1))+W_4(\varnothing),
\]
the five terms corresponding to the first letters $1,2,3,4,5$.  Expanding
by \eqref{eq:W} and \eqref{eq:W-boundary} reaches twenty distinct states in
all, listed here with their values, in increasing order of $\rho_1=p+|L|$:
\begin{center}
\small
\setlength{\tabcolsep}{5pt}
\begin{tabular}{@{}rl@{}}
\toprule
$\rho_1$ & values \\
\midrule
$0$ & $W_0(\varnothing)=1$ \\
$1$ & $W_0((1))=1$, $W_1(\varnothing)=1$ \\
$2$ & $W_0((1,1))=1$, $W_0((2))=2$, $W_1((1))=2$, $W_2(\varnothing)=2$ \\
$3$ & $W_0((1,2))=2$, $W_0((2,1))=2$, $W_0((3))=5$, $W_1((1,1))=3$,
      $W_1((2))=6$, $W_2((1))=6$, $W_3(\varnothing)=6$ \\
$4$ & $W_0((4))=14$, $W_1((3))=20$, $W_2((2))=23$, $W_3((1))=23$,
      $W_4(\varnothing)=23$ \\
$5$ & $W_5(\varnothing)=103$ \\
\bottomrule
\end{tabular}
\end{center}
For example,
\begin{align*}
 W_0((4))&=2W_0((3))+W_0((1,2))+W_0((2,1))=10+2+2=14,\\
 W_1((3))&=W_0((3))+2W_1((2))+W_1((1,1))=5+12+3=20,\\
 W_2((2))&=W_0((3))+W_1((2))+2W_2((1))=5+6+12=23,\\
 W_3((1))&=W_0((3))+W_1((2))+W_2((1))+W_3(\varnothing)=5+6+6+6=23,
\end{align*}
so that $W_5(\varnothing)=14+20+23+23+23=103$.  In the first line the terms
are the two endpoints and the interior choices $(a,b)=(1,2),(2,1)$, and in
the second the new minimum, which joins nothing to the head, the two
endpoints, and the interior choice $(1,1)$.  The values $W_0((\ell))$ are the
Catalan numbers $C_\ell=\frac{1}{\ell+1}\binom{2\ell}{\ell}$, as
\cref{lem:first-letter} predicts.  The number of distinct states grows exponentially with $n$
(\cref{prop:scalar-exponential} below).  \Cref{ex:scalar-kernel} recomputes
the same value from the kernel rows of the shrunk state space.
\end{example}

\subsection{The exponential literal state space}

The recurrence \eqref{eq:W} is finite for each $n$, but memoizing its
composition-valued argument does not give a polynomial bound.

\begin{proposition}\label{prop:scalar-exponential}
For $n\geq2$, the computation of $W_n(\varnothing)$ reaches at least
$F_{n-1}$ distinct composition arguments, where $F_{n-1}$ is the $(n-1)$st
Fibonacci number, $F_1=F_2=1$.
\end{proposition}

\begin{proof}
Choosing the letter $1$ at the initial state reaches $W_0((n-1))$.  Repeated
interior choices reach every positive composition $L$ satisfying
\begin{equation}\label{eq:scalar-fib-comps}
 |L|+\operatorname{len}(L)=n.
\end{equation}
Reversing an interior split replaces
$(a,b,\ldots)$ by $(a+b+1,\ldots)$, and repeated reversals return every such
composition to $(n-1)$.  If $L$ has $s$ parts, then $|L|=n-s$, so the number
of choices is $\binom{n-s-1}{s-1}$.  Summing over
$1\leq s\leq\lfloor n/2\rfloor$ gives
\[
 \sum_{s=1}^{\lfloor n/2\rfloor}\binom{n-s-1}{s-1}=F_{n-1}.\qedhere
\]
\end{proof}

Literal memoization of \eqref{eq:W} therefore has exponentially many
states.  The values of the ordered parts cannot be replaced by their total
and their number.  For example, when $p=0$, \eqref{eq:stack-form} gives
$W_0((1,3))=C_1C_3=5$ and $W_0((2,2))=C_2^2=4$.  The next construction removes the entire
composition from the memoization key while retaining its effect on completion
counts.

\subsection{Protected-tail factorization}

Before the general statement, consider the smallest nontrivial instance of
\eqref{eq:W}: one unread letter below the current minimum and a head of size
one, with an arbitrary tail $L'$.  Either the head letter is read first, which leaves the stack $L'$ at
control $1$, or the lower letter is read first and the head letter next,
which leaves $L'$ at control $0$.  Therefore
\[
  W_1((1)L')=W_1(L')+W_0(L').
\]
The two coefficients count the ways to reach the stack $L'$ at
controls $1$ and $0$, and they do not depend on $L'$.  Writing them as
$K_1(1,1)=K_1(1,0)=1$, the identity reads
$W_1((1)L')=\sum_{t\in\{0,1\}}K_1(1,t)W_t(L')$.  The theorem below shows
that such a factorization holds for every head size: the transitions before
the tail is exposed leave the tail unchanged and do not depend on it.

For a stack $UL$ we write $U\mid L$ when the suffix $L$ is the one under
consideration.  The marker is a notational device and not part of the
state.  A state is an \emph{exposure} of $L$ if its stack is exactly $L$,
that is, if it has the form $(t,\varnothing\mid L)$.  A \emph{path} is a
finite sequence of consecutive transitions.  For a nonempty list $U$, a path
from $(p,U\mid L)$ is \emph{stopped} if it ends at an exposure of $L$ and
contains no earlier one, and $K_U^L(p,t)$ denotes the number of stopped
paths from $(p,U\mid L)$ to $(t,\varnothing\mid L)$.  We write
$K_\ell^L=K_{(\ell)}^L$, and we let $K_\varnothing$ be the identity kernel
$\one_{p=t}$, which is $1$ if $p=t$ and $0$ otherwise.  Kernels are
multiplied as matrices, $(AB)(p,t)=\sum_{u\geq0}A(p,u)B(u,t)$.

\begin{theorem}[Protected-tail factorization]
\label{thm:protected-tail-principle}
For every nonempty list $U$ of positive integers, the number $K_U^L(p,t)$ is
independent of $L$.  We write $K_U(p,t)$ for it.  Then, for all lists $U$,
$V$ and $L$ of positive integers,
\begin{equation}\label{eq:scalar-factorization}
 W_p(UL)=\sum_{t\geq0}K_U(p,t)W_t(L)
 \qquad\hbox{and}\qquad K_{UV}=K_UK_V.
\end{equation}
In particular $K_{(\ell_1,\ldots,\ell_s)}=K_{\ell_1}\cdots K_{\ell_s}$, and
for $L=(\ell_1,\ldots,\ell_s)$
\begin{equation}\label{eq:abstract-factorization}
 W_p(L)=\sum_{t\geq0}(K_{\ell_1}K_{\ell_2}\cdots K_{\ell_s})(p,t)W_t(\varnothing).
\end{equation}
\end{theorem}

\begin{proof}
Both identities are trivial when $U$ is empty, and the second also when
$V$ is empty, so let $U$ and $V$ be nonempty.  Before $L$ is exposed, a state has the form
$(p',U'\mid L)$ with $U'\ne\varnothing$.  By \eqref{eq:W}, its transitions
change only the control and the blocks before the marker, and they and their
targets are the same for every choice of $L$.  Replacing $L$ by another list
therefore maps the stopped paths from $(p,U\mid L)$ bijectively onto those
for the other list, with the same controls, so $K_U^L(p,t)$ is independent
of $L$.  By \cref{prop:scalar-literal}, $W_p(UL)$ is the number of maximal
transition sequences from $(p,UL)$.  Each of them ends at $(0,\varnothing)$,
whose stack is empty, and changes only the blocks before the marker until
$L$ is exposed, so it passes through a first exposure
$(t,\varnothing\mid L)$ of $L$.  Cutting it there gives a stopped path
counted by $K_U(p,t)$ followed by a maximal sequence from $(t,L)$, counted
by $W_t(L)$, and concatenation is the inverse operation.  In the same way, a
stopped path from $(p,UV\mid L)$ leaves $VL$ unchanged until $U$ is used up,
so it passes through a first exposure $(u,\varnothing\mid VL)$ of $VL$, and cutting it there gives a stopped path
from $(p,U\mid VL)$ followed by one from $(u,V\mid L)$, so
$K_{UV}=K_UK_V$.  Taking $U=(\ell_1,\ldots,\ell_s)$ and $L=\varnothing$
in the first identity and applying the second repeatedly gives
\eqref{eq:abstract-factorization}.
\end{proof}

\begin{remark}[Pushdown systems and pop sequences]
\label{rem:pop-sequences}
The proof uses two properties of \eqref{eq:W}: before the marked tail is
exposed, the transitions depend only on the control and the blocks before
the marker and leave the tail unchanged, and every maximal transition
sequence exposes the tail.  They say that the recurrence is a pushdown
system.  The kernel $K_\ell(p,t)$ counts the runs that pop the block $\ell$
while passing from $p$ to $t$, and \cref{thm:protected-tail-principle} is the
counting form of the classical conversion of a pushdown automaton into a
context-free grammar, whose variable $[p\,\ell\,t]$ describes exactly those
runs, see Hopcroft, Motwani, and
Ullman~\cite[Theorem~6.14]{HopcroftMotwaniUllman2001}, and
Sipser~\cite[Lemma~2.27]{Sipser2013} for a variant with variables $A_{pq}$.  The same runs are the pop sequences
of the weighted pushdown systems of Reps, Schwoon, Jha, and
Melski~\cite{RepsSchwoonJhaMelski2005}, and
\eqref{eq:abstract-factorization} is the first-passage decomposition of
weighted lattice paths, see Flajolet and
Sedgewick~\cite[Sections~I.5.3 and~V.4]{FlajoletSedgewick2009}.  With
infinitely many controls the construction gives an infinite grammar, so the
Chomsky--Sch\"utzenberger theorem~\cite{ChomskySchutzenberger1963} does not
apply and no algebraicity follows, and what the factorization provides is a
memoization key of polynomial size (\cref{thm:main}) rather than a
generating function.  The same two properties hold for the recurrence of
general $d$, and \cref{sec:kernels} applies the proof there.
\end{remark}

\subsection{The scalar kernel algorithm}\label{sec:scalar-kernel}

Recall that the control $p$ is the number of unread letters below the
current minimum.  For a kernel entry $K_\ell(p,t)$ of
\cref{thm:protected-tail-principle} we call $p$ the \emph{source control}
and $t$ the \emph{terminal control}, and we call $K_\ell(p,\mathord\cdot)$ a
\emph{scalar kernel row}.  A transition that selects an unread letter below the current
minimum is a \emph{base move}.  A transition that selects a value from the
active head is an \emph{active-head move}.  Put $K_0=K_\varnothing$, the
identity kernel, and $G_p=W_p(\varnothing)$.

\begin{proposition}[Scalar kernel recurrences]
\label{prop:scalar-kernel-recurrence}
For $\ell\geq1$,
\begin{align}
K_\ell(p,t)={}&\sum_{h=0}^{p-1}K_{\ell+p-1-h}(h,t)
 +\sum_{\substack{a,b\geq0\\a+b=\ell-1}}
  \sum_{u\geq0}K_a(p,u)K_b(u,t). \label{eq:scalar-K}
\end{align}
The empty-stack values satisfy
\begin{equation}\label{eq:scalar-G}
 G_p=
 \begin{cases}
  1,&p=0,\\
  \sum_{h=0}^{p-1}\sum_{t\geq0}K_{p-1-h}(h,t)G_t,&p\geq1,
 \end{cases}
 \qquad \lvert\Av_n(1342)\rvert=G_n.
\end{equation}
\end{proposition}

\begin{proof}
We partition the stopped paths from $(p,(\ell)\mid L)$ by their first move.
A new minimum leaving $h$ lower letters gives the first sum in
\eqref{eq:scalar-K}.  A letter of the active head of local rank $r$ leaves
the stack $\nz(a,b)L$ with $a=r-1$ and $b=\ell-r$, so the rest of the path is
a stopped path from $(p,\nz(a,b)\mid L)$, or the empty path when
$\ell=1$.  These are counted by
$K_{\nz(a,b)}(p,t)=(K_aK_b)(p,t)$, by
\cref{thm:protected-tail-principle}, which gives the double sum.  Finally, the recurrence in \eqref{eq:scalar-G}
follows from \eqref{eq:W-boundary}: a new minimum leaving $h$ lower letters
creates a part of size $p-1-h$, to which \eqref{eq:scalar-factorization}
applies, and $K_0$ covers the case $h=p-1$.  The last equation is
\eqref{eq:W-initial}.
\end{proof}

\begin{lemma}[Exact scalar support]\label{lem:scalar-support}
For every $p\geq0$ and $\ell\geq1$, the support of the row
$K_\ell(p,\mathord\cdot)$, the set of $t$ with $K_\ell(p,t)\neq0$, is
\begin{equation}\label{eq:scalar-support}
 \supp K_\ell(p,\mathord\cdot)=\{0,1,\ldots,p\}.
\end{equation}
In addition, $K_\ell(p,p)=C_\ell$.
\end{lemma}

\begin{proof}
A base move strictly decreases the control, while an active-head move leaves it
unchanged.  Thus every terminal control is at most $p$.  A path ending at
$p$ uses no base move.  By \cref{lem:first-letter}, such paths are in
bijection with the $231$-avoiding permutations of the original
$\ell$-element interval, so there are $C_\ell$ of them.

For $t<p$, consider the path that first reads the unread letter below the
current minimum that has $t$ unread letters below it, which makes the
control $t$ and the active head of size $\ell+p-1-t$, and then reads
endpoints until the head is exhausted.  It makes no other base move and
exposes the marked suffix at control $t$.  Thus every $t\in\{0,\ldots,p\}$ occurs.
\end{proof}

We call $p+\ell$ the \emph{grade} of the kernel row $K_\ell(p,\mathord\cdot)$.

\begin{example}[From kernel rows to $\lvert\Av_5(1342)\rvert$]
\label{ex:scalar-kernel}
We write each row $K_\ell(p,\mathord\cdot)$ as the vector
$(K_\ell(p,0),\ldots,K_\ell(p,p))$, which by \cref{lem:scalar-support} lists
its whole support.  Equation \eqref{eq:scalar-K} gives
\[
 K_1(1,\mathord\cdot)=(1,1),\quad
 K_3(0,\mathord\cdot)=(5),\quad
 K_2(1,\mathord\cdot)=(4,2),\quad
 K_1(2,\mathord\cdot)=(3,1,1),
\]
for example
$K_1(2,\mathord\cdot)=K_2(0,\mathord\cdot)+K_1(1,\mathord\cdot)+\mathbf e_2
=(2)+(1,1)+\mathbf e_2=(3,1,1)$, where $\mathbf e_t(u)=\one_{u=t}$.  The
rows of grade $4$ follow:
\begin{align*}
 K_4(0,\mathord\cdot)&=2K_3(0,\mathord\cdot)
  +K_1(0,0)K_2(0,\mathord\cdot)+K_2(0,0)K_1(0,\mathord\cdot)=(10)+(2)+(2)=(14),\\
 K_3(1,\mathord\cdot)&=K_3(0,\mathord\cdot)+2K_2(1,\mathord\cdot)
  +\textstyle\sum_uK_1(1,u)K_1(u,\mathord\cdot)=(5)+(8,4)+(2,1)=(15,5),\\
 K_2(2,\mathord\cdot)&=K_3(0,\mathord\cdot)+K_2(1,\mathord\cdot)
  +2K_1(2,\mathord\cdot)=(5)+(4,2)+(6,2,2)=(15,4,2),\\
 K_1(3,\mathord\cdot)&=K_3(0,\mathord\cdot)+K_2(1,\mathord\cdot)
  +K_1(2,\mathord\cdot)+\mathbf e_3=(12,3,1,1).
\end{align*}
Since $(G_0,\ldots,G_4)=(1,1,2,6,23)$, \eqref{eq:scalar-G} gives
\begin{align*}
G_5
 &=14G_0+(15G_0+5G_1)+(15G_0+4G_1+2G_2)+(12G_0+3G_1+G_2+G_3)+G_4\\
 &=14+20+23+23+23=103,
\end{align*}
the five groups corresponding, in order, to $h=0,1,2,3,4$, the last through
the identity kernel $K_0$.  The five group sums are the five terms
$W_0((4)),\ldots,W_4(\varnothing)$ of \cref{ex:1342-literal-count}: each
kernel row compresses the subtree of the literal recurrence below one
first-letter choice into a vector indexed by the control at which the empty
tail is exposed.
\end{example}

\begin{theorem}[Polynomial enumeration of $1342$-avoiders]
\label{thm:scalar-algorithm}
For every $N\geq1$, the numbers
$\lvert\Av_n(1342)\rvert$ for $0\leq n\leq N$ can be computed using $O(N^5)$
integer arithmetic operations and $O(N^3)$ stored integers.  Every
integer has $O(N\log N)$ bits.  Therefore the bit complexity is
$O(N^5M(N\log N))$ and the bit storage is $O(N^4\log N)$, where $M(b)$ is
the cost of multiplying $b$-bit integers, as in \cref{thm:main}.
\end{theorem}

\begin{proof}
We order the kernel rows by grade.  The row
$K_{\ell+p-1-h}(h,\mathord\cdot)$ in a base term of \eqref{eq:scalar-K} has
grade $p+\ell-1$.  In a split term, $p+a<p+\ell$.  If
$K_a(p,u)\neq0$, then $u\leq p$, by \cref{lem:scalar-support} when $a\geq1$
and since $K_0$ is the identity when $a=0$, so that $u+b\leq p+b<p+\ell$.
Thus increasing grade
is a topological order for all kernel dependencies: an order in which every
row comes after all the rows its recurrence uses, so that evaluating the rows
in this order never meets a row not yet computed.  Equation
\eqref{eq:scalar-G} is then evaluated in increasing $p$, because every
terminal control $t$ in its nonzero terms satisfies
$t\leq h<p$.

We compute every row in the \emph{padded region} $p\geq0$, $\ell\geq1$,
$p+\ell\leq N$: it is closed under the dependencies in the preceding
paragraph and contains every row with $\ell\geq1$ needed to evaluate
\eqref{eq:scalar-G} for $n\leq N$, together with rows that no such
evaluation reaches.  It contains
$\binom{N+1}{2}=O(N^2)$ source rows.  By
\cref{lem:scalar-support}, these rows have at most
\[
 \sum_{\ell=1}^N\sum_{p=0}^{N-\ell}(p+1)
 =\binom{N+2}{3}=O(N^3)
\]
nonzero entries in total.  For one source row, a sparse evaluation
of all split terms has $O(N)$ choices of $a$, $O(N)$ possible intermediate
controls $u$, and $O(N)$ terminal controls in the row being added.  It
therefore uses $O(N^3)$ operations per source row and $O(N^5)$ operations in
total.  The nonsplit kernel terms use $O(N^4)$ operations in total, and
\eqref{eq:scalar-G} uses $O(N^3)$.

For a scalar kernel entry, we take the marked suffix to be empty, by
\cref{thm:protected-tail-principle}.  Its initial number of unread letters is then
$p+\ell\leq N$.  An empty-stack computation also starts with at most $N$
unread letters.  This number decreases at every transition.  At each state,
the number of outgoing transitions is the number of legal next letters and
is at most $N$.  Thus each stored count is
at most $(N+1)^N$ and has $O(N\log N)$ bits.  Accounting for integer
multiplication and for the bit length of every stored entry gives the stated
bit bounds.  Every product and partial accumulator is a nonnegative
contribution to a final kernel or empty-stack value, so the same bound applies
to intermediate integers.
\end{proof}

B\'ona~\cite{Bona1997} found the algebraic generating function
\begin{equation}\label{eq:bona-1342}
 \sum_{n\geq0}\lvert\Av_n(1342)\rvert x^n
 =\frac{(1-8x)^{3/2}-8x^2+20x+1}{2(1+x)^3}.
\end{equation}
\Cref{thm:scalar-algorithm} is a separate transfer construction, and it
extends to the family: we replace $p$ by $d$ ordered band counts, while the
interval stack, its merger rule, and the protected-tail theorem remain
unchanged.

\section{From one threshold to many}
\label{sec:frontier}

We now consider the pattern $\beta_d$ for an arbitrary $d\geq1$, and replace
the current minimum of \cref{sec:obligations} by $d$ thresholds and the
bands of unread values below and between them.  After a
prefix has been read, let $b_j$ be the smallest final entry of an
increasing subsequence of length $j$ contained in that prefix, for
$1\leq j\leq d$.  When the prefix contains no such subsequence, we use the
virtual value $b_j=n+j$.  Then
\[
 b_1<b_2<\cdots<b_d,
\]
since deleting the last entry of an increasing $(j+1)$-subsequence ending at
a nonvirtual $b_{j+1}$ leaves one of length $j$ ending below $b_{j+1}$, and a
virtual $b_{j+1}=n+j+1$ exceeds $b_j\leq n+j$.
We call $b_1,\ldots,b_d$ the \emph{patience-sorting thresholds} of the
prefix, since they are the top cards of the first $d$ piles when the
prefix is dealt by patience sorting, see Aldous and
Diaconis~\cite{AldousDiaconis1999}.  Every nonvirtual threshold value
has already been read.  We call the unread values below $b_d$ the
\emph{base values}.  They fall into $d$ open bands, shown in \cref{fig:bands}.  Band $0$ lies below
$b_1$, while band $i$ lies between $b_i$ and $b_{i+1}$ for $1\leq i<d$.
We write $\mathbf p=(p_0,p_1,\ldots,p_{d-1})$ for the band sizes.  We again
call $\mathbf p$ the \emph{control}, write $\|\mathbf p\|_1=p_0+\cdots+p_{d-1}$
for its \emph{control mass}, and $\mathbf0$ for the zero control.  Let $B_0,\ldots,B_{d-1}$ be the corresponding unread blocks.
Together with the intervals of the stack, the value order is
\begin{equation}\label{eq:ordered-layout}
 B_0<b_1<B_1<b_2<\cdots<b_{d-1}<B_{d-1}<b_d
 <I_1<\cdots<I_s
\end{equation}
by \cref{prop:state-invariant}(a) below.  Empty
blocks are allowed.  As in the one-threshold case (\cref{sec:obligations}), a
base move selects a value from some $B_i$ and an active-head move
selects a value from $I_1$.  We call a base move from $B_i$ with $i<d-1$ an
\emph{early-band move} and a base move from $B_{d-1}$ a \emph{last-band
move}.  Initially
$\mathbf p=(n,0,\ldots,0)$.

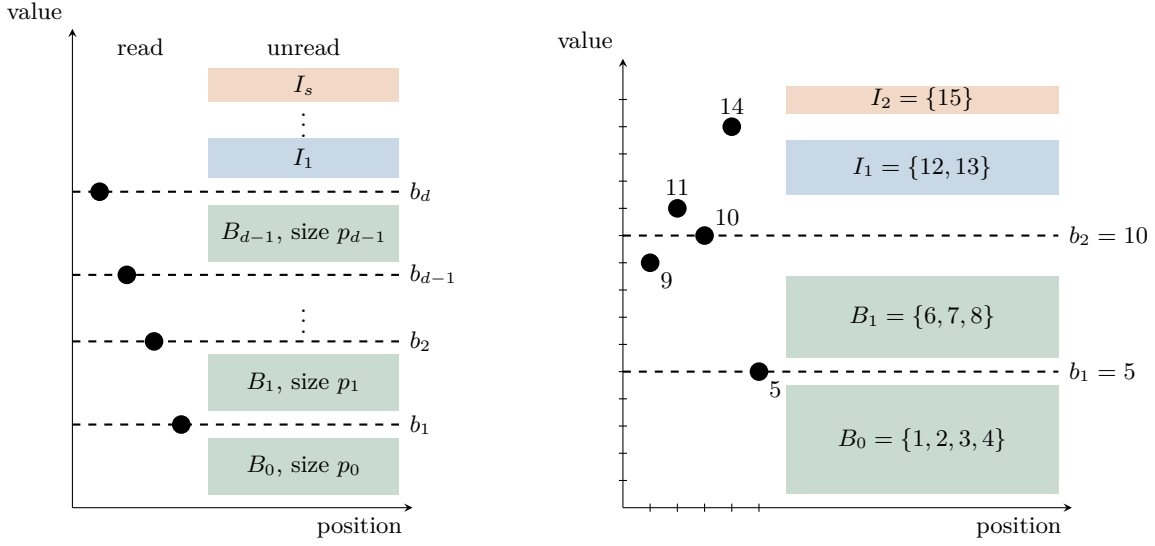
\begin{figure}[tb]
\centering
\begin{minipage}[b]{0.41\textwidth}
\centering
\begin{tikzpicture}[x=0.36cm,y=0.44cm,>=stealth]
  \draw[->] (0,0) -- (12.5,0) node[below left] {\footnotesize position};
  \draw[->] (0,0) -- (0,14.4) node[above left] {\footnotesize value};
  \fill[darkgreen!25] (5,0.4) rectangle (12,2.1);
  \node at (8.5,1.25) {\footnotesize $B_0$, size $p_0$};
  \fill[darkgreen!25] (5,2.9) rectangle (12,4.6);
  \node at (8.5,3.75) {\footnotesize $B_1$, size $p_1$};
  \node at (8.5,5.85) {\footnotesize $\vdots$};
  \fill[darkgreen!25] (5,7.4) rectangle (12,9.1);
  \node at (8.5,8.25) {\footnotesize $B_{d-1}$, size $p_{d-1}$};
  \fill[pairblue!25] (5,9.9) rectangle (12,11.1);
  \node at (8.5,10.5) {\footnotesize $I_1$};
  \node at (8.5,11.75) {\footnotesize $\vdots$};
  \fill[tailorange!25] (5,12.2) rectangle (12,13.2);
  \node at (8.5,12.7) {\footnotesize $I_s$};
  \draw[thick,dashed] (0,9.5) -- (12,9.5) node[right] {\footnotesize $b_d$};
  \draw[thick,dashed] (0,7) -- (12,7) node[right] {\footnotesize $b_{d-1}$};
  \draw[thick,dashed] (0,5) -- (12,5) node[right] {\footnotesize $b_2$};
  \draw[thick,dashed] (0,2.5) -- (12,2.5) node[right] {\footnotesize $b_1$};
  \fill (1,9.5) circle (3.5pt);
  \fill (2,7) circle (3.5pt);
  \fill (3,5) circle (3.5pt);
  \fill (4,2.5) circle (3.5pt);
  \node[anchor=south,font=\footnotesize] at (2.5,13.3) {read};
  \node[anchor=south,font=\footnotesize] at (8.5,13.3) {unread};
\end{tikzpicture}
\end{minipage}%
\hfill
\begin{minipage}[b]{0.57\textwidth}
\centering
\begin{tikzpicture}[x=0.36cm,y=0.36cm,>=stealth]
  \draw[->] (0,0) -- (16.5,0) node[below left] {\footnotesize position};
  \draw[->] (0,0) -- (0,16.5) node[above left] {\footnotesize value};
  \foreach \y in {1,...,15} \draw (-0.15,\y) -- (0.15,\y);
  \foreach \x in {1,...,5} \draw (\x,-0.15) -- (\x,0.15);
  \fill[darkgreen!25] (6,0.5) rectangle (16,4.5);
  \node at (11,2.5) {\footnotesize $B_0=\{1,2,3,4\}$};
  \fill[darkgreen!25] (6,5.5) rectangle (16,8.5);
  \node at (11,7) {\footnotesize $B_1=\{6,7,8\}$};
  \fill[pairblue!25] (6,11.5) rectangle (16,13.5);
  \node at (11,12.5) {\footnotesize $I_1=\{12,13\}$};
  \fill[tailorange!25] (6,14.5) rectangle (16,15.5);
  \node at (11,15) {\footnotesize $I_2=\{15\}$};
  \draw[thick,dashed] (0,5) -- (16,5) node[right] {\footnotesize $b_1=5$};
  \draw[thick,dashed] (0,10) -- (16,10) node[right] {\footnotesize $b_2=10$};
  \fill (1,9) circle (3.5pt);  \node[below right=0pt] at (1,9) {\footnotesize $9$};
  \fill (2,11) circle (3.5pt); \node[above=1pt] at (2,11) {\footnotesize $11$};
  \fill (3,10) circle (3.5pt); \node[above right=0pt] at (3,10) {\footnotesize $10$};
  \fill (4,14) circle (3.5pt); \node[above=1pt] at (4,14) {\footnotesize $14$};
  \fill (5,5) circle (3.5pt);  \node[below right=0pt] at (5,5) {\footnotesize $5$};
\end{tikzpicture}
\end{minipage}
\caption{The bands and the stack.  Left, schematically: the thresholds
$b_1<\cdots<b_d$ (the only read letters drawn, at arbitrary positions),
the bands $B_0,\ldots,B_{d-1}$ of unread base values with sizes
$\mathbf p=(p_0,\ldots,p_{d-1})$, and the stack $I_1,\ldots,I_s$ above
$b_d$.  Right: the running example \eqref{eq:running-example} for $d=2$
after the prefix $9,11,10,14,5$, with $b_1=5$, $b_2=10$,
$\mathbf p=(4,3)$, and the stack $\{12,13\}\mid\{15\}$, as in row $5$ of the
table in \cref{ex:full-state}.}
\label{fig:bands}
\end{figure}

If the next base value $x$ lies in band $i\geq1$, then $x$ ends an increasing
subsequence of length $i+1$, since $b_i<x$.  For $i=0$, the letter $x$ is an
increasing subsequence of length one.  In either case, $x$ ends no increasing
subsequence of length $i+2$, since that would imply $b_{i+1}<x$, while
$x\in B_i$ gives $x<b_{i+1}$.  Thus only
$b_{i+1}$ changes.  A base value creates a new $d$-trigger exactly when it
lies in band $d-1$.  The first $d-1$ bands update one threshold without
creating an obligation.  The last band also creates a $231$-avoidance
obligation.

Since a $d$-trigger is the last entry of an increasing $d$-subsequence,
$b_d$ is the least $d$-trigger of the prefix when it is nonvirtual, and a
virtual $b_d=n+d$ means that the prefix has no $d$-trigger.  With $q$ as in
\cref{def:scan-state}, in both cases
\begin{equation}\label{eq:bd-q}
 b_d=q.
\end{equation}
In particular, the unread values below $b_d$ belong to the projection of no
trigger read so far, while every unread value above a nonvirtual $b_d$
belongs to the projection created by the trigger $b_d$.  In the terms of
\cref{def:scan-state}, a value of $B_i$ with $i<d-1$ is a non-trigger move,
a value of $B_{d-1}$ is a merger, and a value of $I_1$ is a split
(\cref{prop:state-invariant}(b) below).

For example, in the permutation $316829574$ the entries $3,6$ form an increasing $2$-subsequence, and the subsequent
projection onto values larger than $6$ is $8,9,7$, an occurrence of $231$.
Therefore
$3,6,8,9,7$ is an occurrence of $\beta_2=12453$.  The interspersed values
$2,5,4$, all smaller than $6$, do not belong to this projection.

\section{The literal family recurrence}\label{sec:literal}

In this section we write down the literal recurrence for $\beta_d$, show that
it counts every avoider once, characterize the separations of its stack, and
show that its state space is exponential.  We now express the band updates in terms of their sizes.  For $0\leq i<d-1$ and
$0\leq h<p_i$, let $x$ be the value of $B_i$ that has $h$ unread values of
$B_i$ below it.  Reading $x$ makes it the new threshold $b_{i+1}$, leaves
$h$ values in band $i$, and moves the $p_i-1-h$ values of $B_i$ above it to
band $i+1$, so the control becomes
\begin{equation}\label{eq:T}
 T_{i,h}(\mathbf p)=
 (p_0,\ldots,p_{i-1},h,
  p_{i+1}+p_i-1-h,p_{i+2},\ldots,p_{d-1}).
\end{equation}
Similarly, if $x$ is the value of $B_{d-1}$ that has $h$ unread values of
$B_{d-1}$ below it, reading $x$ gives the control
\begin{equation}\label{eq:U}
 U_h(\mathbf p)=(p_0,\ldots,p_{d-2},h),
 \qquad \delta_h=p_{d-1}-1-h.
\end{equation}
For $d=1$, interpret \eqref{eq:U} as $U_h((p_0))=(h)$.
Here $x$ is a new least $d$-trigger, and the $\delta_h$ unread
values above it and below the old $b_d$ join the active head.

For example, take $d=2$ and the prefix $9,11,10,14,5$ of the running example
\eqref{eq:running-example}, shown on the right of \cref{fig:bands}, with
$b_1=5$, $b_2=10$, and $\mathbf p=(4,3)$.  Reading $2$, which has $h=1$
unread value of $B_0=\{1,2,3,4\}$ below it, makes $2$ the new $b_1$, leaves
$\{1\}$ in band $0$, and moves $3,4$ into band $1$, so the control becomes
$T_{0,1}(\mathbf p)=(1,5)$.  Reading $7$ instead, which has $h=1$ unread
value of $B_1=\{6,7,8\}$ below it, makes $7$ the new $b_2$ and gives the
control $U_1(\mathbf p)=(4,1)$, and the $\delta_1=1$ value $8$ joins the
head, so the stack becomes $\{8,12,13\}\mid\{15\}$.  \Cref{ex:full-state}
follows the whole running example through such moves.

Let $L=(\ell_1,\ldots,\ell_s)$ be the sizes of the intervals of the stack,
active head first, and call $(\mathbf p,L)$ the \emph{state} of a legal
prefix (\cref{def:scan-state}).  For an empty stack
we interpret $\ell_1=0$ and its tail as empty.

\begin{proposition}[State invariant]\label{prop:state-invariant}
Let $\sigma$ be a legal prefix with thresholds $b_1<\cdots<b_d$, band sizes
$\mathbf p$, and stack $\mathcal S(\sigma)=(I_1,\ldots,I_s)$ with sizes
$L=(\ell_1,\ldots,\ell_s)$.
\begin{enumerate}
\item[(a)] The stack consists of the unread values above $b_d$, so the
  ordered layout \eqref{eq:ordered-layout} holds.
\item[(b)] The legal moves from $\sigma$, and their effect on
  $(\mathbf p,L)$, are the following.  A value of $B_i$ with $i<d-1$ having
  $h$ unread values of $B_i$ below it is a non-trigger move to
  $(T_{i,h}(\mathbf p),L)$.  A value of $B_{d-1}$ having $h$ unread values of
  $B_{d-1}$ below it is a merger to
  $(U_h(\mathbf p),\nz(\ell_1+\delta_h,\ell_2,\ldots))$.  A value of $I_1$
  of local rank $r$ is a split to
  $(\mathbf p,\nz(r-1,\ell_1-r,\ell_2,\ldots))$.  No other letter is legal.
\end{enumerate}
\end{proposition}

\begin{proof}
(a) Since $b_d=q$ by \eqref{eq:bd-q}, the stack consists of the unread
values above $b_d$ by \cref{lem:legal-moves}(a), while the unread values
below $b_d$ form the bands by definition.  This is \eqref{eq:ordered-layout}.

(b) We classify the unread letters by \cref{def:scan-state}.  As observed
before \eqref{eq:bd-q}, a letter $x\in B_i$ changes only $b_{i+1}$, and it
is a non-trigger move if $i<d-1$ and a merger if $i=d-1$, since then
$x<b_d=q$.  A letter of $I_1$ exceeds $b_d$, so it changes no threshold, and
it is a split, as noted after \cref{def:scan-state}.  The new controls are
given by \eqref{eq:T} and \eqref{eq:U}, the set $E$ of a merger has
$\delta_h$ elements, and the new stacks follow from \cref{def:scan-state}.
Every other unread letter lies in a deferred interval and is illegal by
\cref{lem:legal-moves}(b).
\end{proof}

For a control $\mathbf p$ and a list $L$ of positive integers with
$(\mathbf p,L)\ne(\mathbf0,\varnothing)$, we define
\begin{align}
H_{\mathbf p}(L)={}&
 \sum_{i=0}^{d-2}\sum_{h=0}^{p_i-1}
 H_{T_{i,h}(\mathbf p)}(L)                                           \label{eq:H}\\
&+\sum_{h=0}^{p_{d-1}-1}
 H_{U_h(\mathbf p)}\!\left(
   \nz(\ell_1+\delta_h,\ell_2,\ldots)\right)                       \notag\\
&+\sum_{\substack{a,b\geq0\\a+b=\ell_1-1}}
 H_{\mathbf p}\!\left(\nz(a,b,\ell_2,\ldots)\right).                \notag
\end{align}
The three lines represent, respectively, an early-band move, a last-band
move, and a letter of the active head, of local rank $a+1$, and as in
\cref{sec:one-threshold} each summand stands for a transition.  The first
double sum is empty when $d=1$, and the last line is empty when the stack
is empty.  The quantity
\begin{equation}\label{eq:rho}
 \rho(\mathbf p,L)=\|\mathbf p\|_1+\sum_{j=1}^{s}\ell_j
\end{equation}
drops by one on every transition.  An early-band move reduces the control
mass by one.  A last-band move reduces it by $\delta_h+1$ and adds
$\delta_h$ to the head.  An active-head move consumes one head value.
The zero state $(\mathbf0,\varnothing)$ is the only state without
transitions: a nonzero control allows a base move and a nonempty stack a
letter of the head.  Therefore \eqref{eq:H} and $H_{\mathbf0}(\varnothing)=1$ define
$H_{\mathbf p}(L)$, and by induction on $\rho$ it is the number of maximal
transition sequences from $(\mathbf p,L)$, all of which end at the zero
state.  Theorem~\ref{thm:literal} below shows that
\begin{equation}\label{eq:initial-terminal}
 a_n^{(d)}:=\lvert\Av_n(\beta_d)\rvert
 =H_{(n,0,\ldots,0)}(\varnothing).
\end{equation}
For $d=1$, identify $H_{(p)}(L)=W_p(L)$.  The early-band sum is empty and
$U_h((p))=(h)$, so \eqref{eq:H} and $H_{\mathbf0}(\varnothing)=1$ reduce to
\eqref{eq:W} and \eqref{eq:W-boundary}, \eqref{eq:rho} to
\eqref{eq:rho-scalar}, and \eqref{eq:initial-terminal} to
\eqref{eq:W-initial}.

\begin{example}[What the second threshold adds]\label{ex:full-state}
We specialize to $d=2$ and scan the running example
\eqref{eq:running-example} again, which can be read down the second column
of the table below.  It avoids $12453$, since $12453$ contains
$1342$.  The table lists, after each letter $\pi_i$, the type of move, the
thresholds $b_1,b_2$, the bands $B_0,B_1$, the interval stack, and
$H_{\mathbf p}(L)$, the number of $12453$-avoiding completions.  A dot marks
a virtual threshold.  Band $0$ is the early band and band $1$ the last band,
and in the sense of \cref{def:scan-state} an early-band move is a
non-trigger move and a last-band move is a merger.  As in \cref{ex:1342-scan}, the table
writes ``merger'' when the move joins values to the head or creates the
stack, so that ``last band'' alone marks the case $E=\varnothing$, and it
writes ``head exhausted'' for an endpoint choice in a head of size one.
\begin{center}
\footnotesize
\setlength{\tabcolsep}{3.5pt}
\begin{tabular}{@{}rrllllll@{}}
\toprule
$i$ & $\pi_i$ & move & $b_1,b_2$ & $B_0$ & $B_1$ & interval stack & $H_{\mathbf p}(L)$ \\
\midrule
$1$ & $9$ & early band & $9,\cdot$ & $\{1,\ldots,8\}$ & $\{10,\ldots,15\}$ & $\varnothing$ & $H_{(8,6)}(\varnothing)$ \\
$2$ & $11$ & last band, merger & $9,11$ & $\{1,\ldots,8\}$ & $\{10\}$ & $\{12,\ldots,15\}$ & $H_{(8,1)}((4))$ \\
$3$ & $10$ & last band & $9,10$ & $\{1,\ldots,8\}$ & $\varnothing$ & $\{12,\ldots,15\}$ & $H_{(8,0)}((4))$ \\
$4$ & $14$ & interior choice & $9,10$ & $\{1,\ldots,8\}$ & $\varnothing$ & $\{12,13\}\mid\{15\}$ & $H_{(8,0)}((2,1))$ \\
$5$ & $5$ & early band & $5,10$ & $\{1,\ldots,4\}$ & $\{6,7,8\}$ & $\{12,13\}\mid\{15\}$ & $H_{(4,3)}((2,1))$ \\
$6$ & $12$ & endpoint choice & $5,10$ & $\{1,\ldots,4\}$ & $\{6,7,8\}$ & $\{13\}\mid\{15\}$ & $H_{(4,3)}((1,1))$ \\
$7$ & $6$ & last band, merger & $5,6$ & $\{1,\ldots,4\}$ & $\varnothing$ & $\{7,8,13\}\mid\{15\}$ & $H_{(4,0)}((3,1))$ \\
$8$ & $2$ & early band & $2,6$ & $\{1\}$ & $\{3,4\}$ & $\{7,8,13\}\mid\{15\}$ & $H_{(1,2)}((3,1))$ \\
$9$ & $3$ & last band, merger & $2,3$ & $\{1\}$ & $\varnothing$ & $\{4,7,8,13\}\mid\{15\}$ & $H_{(1,0)}((4,1))$ \\
$10$ & $8$ & interior choice & $2,3$ & $\{1\}$ & $\varnothing$ & $\{4,7\}\mid\{13\}\mid\{15\}$ & $H_{(1,0)}((2,1,1))$ \\
$11$ & $4$ & endpoint choice & $2,3$ & $\{1\}$ & $\varnothing$ & $\{7\}\mid\{13\}\mid\{15\}$ & $H_{(1,0)}((1,1,1))$ \\
$12$ & $7$ & head exhausted & $2,3$ & $\{1\}$ & $\varnothing$ & $\{13\}\mid\{15\}$ & $H_{(1,0)}((1,1))$ \\
$13$ & $1$ & early band & $1,3$ & $\varnothing$ & $\varnothing$ & $\{13\}\mid\{15\}$ & $H_{(0,0)}((1,1))$ \\
$14$ & $13$ & head exhausted & $1,3$ & $\varnothing$ & $\varnothing$ & $\{15\}$ & $H_{(0,0)}((1))$ \\
$15$ & $15$ & head exhausted & $1,3$ & $\varnothing$ & $\varnothing$ & $\varnothing$ & $H_{(0,0)}(\varnothing)$ \\
\bottomrule
\end{tabular}
\end{center}
Compare with the one-threshold scan of \cref{ex:1342-scan}.  The first
letter $9$ is not a $2$-trigger and creates no obligation.  It becomes
$b_1$, and the unread values above it form band $1$.  The first
$2$-trigger is $11$, whose obligation constrains $\{12,\ldots,15\}$.  The
value $10$ remains a base value, and reading it next makes it the least
$2$-trigger without joining anything to the head.  The new minimum $5$, a merger for
$d=1$, is here an early-band move: it transfers $\{6,7,8\}$ from band $0$
to band $1$ and leaves the head $\{12,13\}$ unchanged.  It is the
$2$-trigger $6$, not the base value $5$, that joins $\{7,8\}$ to the
head.  The read letters $9$ to $12$ lie inside the range of the head
$\{7,8,13\}$, and by \cref{cor:separators} below none of them separates $8$
from $13$: neither $9$ nor $11$ was read after an earlier $2$-trigger smaller
than $8$, and $10$ and $12$ were read only after $2$-triggers exceeding $8$.
The letter $14$ separates $13$ from $15$, because it was read after the
$2$-trigger $10<13$.  The letter $8$, an endpoint
choice for $d=1$, is here an interior choice.

After the eighth letter, the thresholds are $b_1=2$ and $b_2=6$, the bands
are $B_0=\{1\}$ and $B_1=\{3,4\}$, and the stack is $\{7,8,13\}\mid\{15\}$.
The state is $(\mathbf p,L)=((1,2),(3,1))$.  Its next choices are:
\begin{center}
\small
\begin{tabular}{@{}c>{\raggedright\arraybackslash}p{6.5cm}l@{}}
\toprule
next & action & successor \\
\midrule
$1$ & update the first threshold & $H_{(0,2)}((3,1))$ \\
$3$ & new $2$-trigger, $\{4\}$ joins the head
    & $H_{(1,0)}((4,1))$ \\
$4$ & new $2$-trigger, no value is added to the head
    & $H_{(1,1)}((3,1))$ \\
$7$ or $13$ & endpoint choice in the active head
    & $H_{(1,2)}((2,1))$, twice \\
$8$ & interior choice with two singleton subintervals
    & $H_{(1,2)}((1,1,1))$ \\
$15$ & illegal until the head is exhausted
    & illegal \\
\bottomrule
\end{tabular}
\end{center}
Therefore
\begin{align*}
 H_{(1,2)}((3,1))={}&H_{(0,2)}((3,1))+H_{(1,0)}((4,1))
 +H_{(1,1)}((3,1))\\
 &+2H_{(1,2)}((2,1))+H_{(1,2)}((1,1,1)).
\end{align*}
The extra threshold is visible in the first row: the choice $1$ updates
$b_1$ without creating a $2$-trigger and, in this state, transfers no unread
value to the second band.
The values of an active head need not be read consecutively.  Base values may
occur between them.  The scan continues with $3$.
\end{example}

\begin{theorem}[Literal recurrence]\label{thm:literal}
Let $d\geq1$.  For every legal prefix with state $(\mathbf p,L)$, reading the
letters of a completion one at a time is a bijection from the completions
that produce a $\beta_d$-avoiding permutation to the maximal sequences of
transitions from $(\mathbf p,L)$, and $H_{\mathbf p}(L)$ is their number.  In
particular \eqref{eq:initial-terminal} holds.
\end{theorem}

\begin{proof}
By \cref{prop:state-invariant}(a), the unread values of a legal prefix with
state $(\mathbf p,L)$ are the $\rho(\mathbf p,L)$ values of its bands and
its stack, and we argue by induction on this number, as in the proof of
\cref{prop:scalar-literal}.  If no value is unread, the prefix is a
permutation, it avoids $\beta_d$ by \cref{prop:scan-states}(a), and its
state is the zero state.  Otherwise, by \cref{prop:scan-states}(b) an
illegal letter begins no avoiding completion, and by
\cref{prop:state-invariant}(b) the legal letters and their successor states
are the terms of \eqref{eq:H}, one letter for each term.  The induction
hypothesis applies to the successors.  The empty prefix has the state
$((n,0,\ldots,0),\varnothing)$.
\end{proof}

For $d=2$, \eqref{eq:H} is the recurrence evaluated by Biers-Ariel's
program~\cite{BiersArielA116485}.  Its function
\path{Avoid12453_helper(i,j,L)} is $H_{(i-1,j-i-1)}(L)$, with the list
$[0]$ for the empty stack, and its four contributions to the output are, in
order, the early-band, last-band, endpoint and interior terms of
\eqref{eq:H}.

The stack can be read off the scanned prefix.  Recall that two unread
values are adjacent if no unread value lies strictly between them.  By
\cref{prop:scan-states}(c), the legal prefixes are exactly the prefixes of
$\beta_d$-avoiding permutations, and every legal prefix has an avoiding
completion, so condition (iii) below is not vacuous.

\begin{corollary}[Separations and completions]\label{cor:separators}
After a legal prefix has been read, let $u<v$ be adjacent unread values
above $b_d$.  The following are equivalent.
\begin{enumerate}
\item[(i)] $u$ and $v$ lie in different intervals of the stack.
\item[(ii)] Some letter $x$ with $u<x<v$ was read after a $d$-trigger
  smaller than $u$.
\item[(iii)] $u$ precedes $v$ in every $\beta_d$-avoiding completion.
\end{enumerate}
In particular, the stack is determined by the scanned prefix: it consists
of the unread values above $b_d$, cut exactly at the adjacent pairs that
satisfy (ii).
\end{corollary}

\begin{proof}
Since $b_d=q$ by \eqref{eq:bd-q}, (i) and (ii) are equivalent by
\cref{lem:separators}.  We prove that (ii) implies (iii) and that (iii)
implies (i).

Suppose (ii) holds, and let $c$ be the trigger.  If $v$ preceded $u$ in a
completion, an increasing $d$-subsequence ending at $c$, followed by
$x,v,u$, would form $\beta_d$, since $c<u<x<v$.  Therefore (iii) holds.

Suppose (i) fails, so that $u$ and $v$ lie in the same interval.  The
avoiding completion constructed in the proof of \cref{prop:scan-states}(c)
reads each interval in decreasing order, and so reads $v$ before $u$.
Therefore (iii) fails.
\end{proof}

Condition (ii) locates the separations, not the letters that produced
them: several letters between $u$ and $v$ may satisfy (ii), and the letter
whose split created the separation need not be among them.  Adjacency is
needed only in (ii), as the example before \cref{ex:1342-scan} shows.  For arbitrary unread $u<v$ above $b_d$,
conditions (i) and (iii) remain equivalent, and they hold exactly when some
pair $u'<v'$ of adjacent unread values with $u\leq u'<v'\leq v$ satisfies
(ii).

\begin{proposition}\label{prop:exponential}
For fixed $d\geq1$ and $n\geq d+1$, literal memoization of \eqref{eq:H} has
at least $F_{n-d}$ distinct composition arguments when computing
$a_n^{(d)}$, where $F_{n-d}$ is the $(n-d)$th Fibonacci number,
$F_1=F_2=1$.
\end{proposition}

\begin{proof}
Reading $1,2,\ldots,d$ first reaches the state $(\mathbf0,(n-d))$: the
first $d-1$ letters are early-band moves, and $d$ is a last-band move whose
$n-d$ larger unread values form the head.  From there the interior choices
in the proof of \cref{prop:scalar-exponential}, with $n$ replaced by
$n-d+1$, reach $F_{n-d}$ distinct compositions.
\end{proof}

The literal state space is therefore exponential in $n$.
\Cref{sec:kernels} removes the stack from the memoization key.

\section{Protected tails and transfer kernels}\label{sec:kernels}

In this section we factor the literal recurrence through transfer kernels and
derive recurrences for the kernels and for the empty-stack values.  As
shown after \eqref{eq:rho}, $H_{\mathbf p}(L)$ is the number of maximal
transition sequences from $(\mathbf p,L)$, for every state.  We mark an arbitrary tail $L$ of the stack
and define exposures, stopped paths and the identity kernel $K_\varnothing$
as in \cref{sec:obligations}.  An early-band move changes only the control.
A last-band move can enlarge the active head, and a letter of the active
head replaces it by its parts below and above that letter.  Thus,
before $L$ is exposed, the transitions change only the control and the
blocks before the marker, and \eqref{eq:H} shows that they and their targets
are the same for every $L$.  Every maximal transition sequence ends at the
zero state, whose stack is empty.  These are the two properties used in the
proof of \cref{thm:protected-tail-principle}, and the same proof gives the
following corollary.  For a nonempty list $U$ we let
$K_U^L(\mathbf p,\mathbf t)$ be the number of stopped paths from
$(\mathbf p,U\mid L)$ to $(\mathbf t,\varnothing\mid L)$, and we write
$K_\ell^L=K_{(\ell)}^L$ and $K_0=K_\varnothing$.

\begin{corollary}[Protected-tail factorization for $\beta_d$]\label{cor:protected-tail}
The number $K_U^L(\mathbf p,\mathbf t)$ is independent of $L$.  We write
$K_U(\mathbf p,\mathbf t)$ for it.  Then, for all lists $U$, $V$ and $L$ of
positive integers,
\begin{equation}\label{eq:factorization}
 H_{\mathbf p}(UL)=
 \sum_{\mathbf t\in\N^d}K_U(\mathbf p,\mathbf t)H_{\mathbf t}(L)
 \qquad\hbox{and}\qquad K_{UV}=K_UK_V.
\end{equation}
In particular, for $L=(\ell_1,\ldots,\ell_s)$,
\begin{equation}\label{eq:matrix-product}
 H_{\mathbf p}(L)=\sum_{\mathbf t\in\N^d}
 (K_{\ell_1}K_{\ell_2}\cdots K_{\ell_s})(\mathbf p,\mathbf t)
 H_{\mathbf t}(\varnothing).
\end{equation}
\end{corollary}

As in \cref{sec:scalar-kernel}, we call $K_\ell(\mathbf p,\mathord\cdot)$ a
\emph{kernel row}, with \emph{source control} $\mathbf p$ and
\emph{terminal controls} $\mathbf t$, and its support is the set of
$\mathbf t$ with $K_\ell(\mathbf p,\mathbf t)\ne0$.  The factorization
replaces the composition argument by a product of transfer kernels,
evaluated on the required source rows.  The recurrence for the kernel rows
follows, and \cref{sec:algorithm} bounds the number of source and terminal
controls required for the values $\lvert\Av_n(\beta_d)\rvert$ with $n\leq N$.

\begin{proposition}[Compressed kernel recurrence]
\label{prop:kernel-recurrence}
For $\ell\geq1$, the transfer kernels satisfy
\begin{align}
K_\ell(\mathbf p,\mathbf t)={}&
 \sum_{i=0}^{d-2}\sum_{h=0}^{p_i-1}
 K_\ell(T_{i,h}(\mathbf p),\mathbf t)                               \label{eq:K}\\
&+\sum_{h=0}^{p_{d-1}-1}
 K_{\ell+\delta_h}(U_h(\mathbf p),\mathbf t)                        \notag\\
&+\sum_{\substack{a,b\geq0\\a+b=\ell-1}}
  \sum_{\mathbf u\in\N^d}K_a(\mathbf p,\mathbf u)
  K_b(\mathbf u,\mathbf t).                                         \notag
\end{align}
\end{proposition}

\begin{proof}
The first-move partition of the proof of \cref{prop:scalar-kernel-recurrence}
applies, with \cref{cor:protected-tail} in place of
\cref{thm:protected-tail-principle}.  Early-band moves give the first sum,
and a last-band move enlarges the head by $\delta_h$.
\end{proof}

We now connect the kernels to an empty stack.  Put
$G_{\mathbf p}=H_{\mathbf p}(\varnothing)$.

\begin{corollary}[Empty-stack recurrence]\label{cor:empty-stack-recurrence}
The empty-stack values satisfy
\begin{equation}\label{eq:G}
G_{\mathbf p}=
\begin{cases}
 1,&\mathbf p=\mathbf0,\\[2pt]
 \displaystyle\sum_{i=0}^{d-2}\sum_{h=0}^{p_i-1}G_{T_{i,h}(\mathbf p)}
 +\sum_{h=0}^{p_{d-1}-1}\sum_{\mathbf t\in\N^d}
 K_{\delta_h}(U_h(\mathbf p),\mathbf t)G_{\mathbf t},&\mathbf p\ne\mathbf0,
\end{cases}
\end{equation}
and
\begin{equation}\label{eq:answer-G}
 a_n^{(d)}=G_{(n,0,\ldots,0)}.
\end{equation}
\end{corollary}

\begin{proof}
This follows from \eqref{eq:H} with an empty stack.  A last-band move creates
a head of size $\delta_h$, to which \eqref{eq:factorization} applies, and
$K_0$ covers the case $\delta_h=0$.  The last equation is
\eqref{eq:initial-terminal}.
\end{proof}

\section{The family algorithm and its complexity}\label{sec:algorithm}

In this section we bound the supports of the kernel rows and count the
operations, stored integers and bit sizes of the resulting algorithm, which
proves \cref{thm:main}.  As in \cref{sec:scalar-kernel}, we give a kernel source
$(\mathbf p,\ell)$ the grade
\[
 w=\|\mathbf p\|_1+\ell.
\]
We verify below that increasing grade is a topological order for the kernel
dependencies in \eqref{eq:K}, and that increasing control mass orders the
empty-stack recurrence \eqref{eq:G}.  The split term is controlled by the
following support bound.

\begin{lemma}[Kernel support]\label{lem:support}
For every $\mathbf p\in\N^d$ and $\ell\geq1$,
\begin{equation}\label{eq:support}
 \supp K_\ell(\mathbf p,\mathord\cdot)
 \subseteq
 \{\mathbf p\}\cup
 \{\mathbf t:\|\mathbf t\|_1\leq\|\mathbf p\|_1-1\}.
\end{equation}
\end{lemma}

\begin{proof}
As in the proof of \cref{lem:scalar-support}, a base move strictly lowers
the control mass and an active-head move leaves the control unchanged.
\end{proof}

In the algorithm below, each row operation is coefficientwise, and each row
is accumulated from zero.

\emph{Kernel phase.}
For $w=1,2,\ldots,N$, and every $\mathbf p\in\N^d$, $\ell\geq1$ with
$\|\mathbf p\|_1+\ell=w$, evaluate \eqref{eq:K} coefficientwise and store
the resulting sparse row.  In each split term, for every $a,b\geq0$ with
$a+b=\ell-1$ and
$\mathbf u\in\supp K_a(\mathbf p,\mathord\cdot)$, add the scalar multiple
$K_a(\mathbf p,\mathbf u)K_b(\mathbf u,\mathord\cdot)$ to the row being
computed.

\emph{Empty-stack phase.}
Set $G_{\mathbf0}=1$.  For $m=1,2,\ldots,N$, and every
$\mathbf p\in\N^d$ with $\|\mathbf p\|_1=m$, evaluate \eqref{eq:G}, using
$K_0$ when $\delta_h=0$.  Return
$(G_{(n,0,\ldots,0)})_{0\leq n\leq N}$.

As in the proof of \cref{thm:scalar-algorithm}, with \cref{lem:support} in
place of \cref{lem:scalar-support}, increasing $w$ is a topological order
for the kernel phase and increasing $m$ for the empty-stack phase.  The new
terms are the early-band rows of \eqref{eq:K}, of grade $w-1$ since
$\|T_{i,h}(\mathbf p)\|_1=\|\mathbf p\|_1-1$, and the last-band rows, of
grade $w-1$ since
\[
 \|U_h(\mathbf p)\|_1+(\ell+\delta_h)
 =\|\mathbf p\|_1-(\delta_h+1)+\ell+\delta_h=w-1.
\]
In \eqref{eq:G}, the early-band controls have mass $m-1$, and by
\cref{lem:support} every terminal control $\mathbf t$ of
$K_{\delta_h}(U_h(\mathbf p),\mathord\cdot)$ has
$\|\mathbf t\|_1\leq\|U_h(\mathbf p)\|_1=m-1-\delta_h$, trivially when
$\delta_h=0$.
The kernel phase computes the rows of the \emph{padded region}
$\|\mathbf p\|_1+\ell\leq N$, $\ell\geq1$, which by stars and bars has
\begin{equation}\label{eq:row-count}
 \binom{N+d}{d+1}=O_d(N^{d+1})
\end{equation}
rows.  By \cref{lem:support}, a row has at most
$1+\binom{\|\mathbf p\|_1+d-1}{d}=O_d(N^d)$ possible terminal controls.

\begin{proof}[Proof of \cref{thm:main}]
Correctness follows from \cref{thm:literal,cor:protected-tail,prop:kernel-recurrence,cor:empty-stack-recurrence}:
the algorithm evaluates \eqref{eq:K} and \eqref{eq:G} in the preceding
topological orders and returns the values \eqref{eq:answer-G}.  We now bound the resources.

There are $O_d(N^{d+1})$ source rows and $O_d(N^d)$ possible entries per
row, giving $O_d(N^{2d+1})$ stored integers.  In a split term there
are $O(N)$ choices of $a$, $O_d(N^d)$ intermediate controls $\mathbf u$,
and $O_d(N^d)$ entries in the row added for each $\mathbf u$.  Thus a
sparse implementation uses $O_d(N^{2d+1})$ operations per source row and
$O_d(N^{3d+2})$ over all kernels.  The nonsplit part uses $O_d(N)$ row
additions per source and therefore $O_d(N^{2d+2})$ operations in total.  For
each of the $O_d(N^d)$ empty-stack controls,
\eqref{eq:G} sums at most $O(N)$ choices over $O_d(N^d)$ terminal controls,
so the $G$ table costs
$O_d(N^{2d+1})$.  For $d\geq1$, the exponents $2d+2$ and $2d+1$ are both
smaller than $3d+2$.

By the argument in the last paragraph of the proof of
\cref{thm:scalar-algorithm}, where a state now has at most
$\rho(\mathbf p,L)\leq N$ unread values and as many legal next letters, every
stored
integer and every intermediate product and partial sum is at most
$(N+1)^N$ and has $O(N\log N)$ bits.  Replacing integer multiplication by its
$M(N\log N)$ bit cost proves the stated bit complexity.  The bit storage is
$O_d(N^{2d+2}\log N)$.
\end{proof}

\section{The two-threshold quotient for
\texorpdfstring{$12453$}{12453}}\label{sec:d2-support}

In this section we specialize to $d=2$ and write $\mathbf p=(p,q)$, so
that $q$ is the size of band~$1$, not the least $2$-trigger of
\eqref{eq:bd-q}.  Two symmetries of the kernels are
available in this case, and they serve different purposes.  The first is a
translation in the first control coordinate: a kernel entry depends on the
first coordinates of its source and terminal controls only through their
difference, so one coordinate can be dropped from the table.  This is the
symmetry that improves the $d=2$ specialization $O(N^8)$ and $O(N^5)$
of \cref{thm:main} to the $O(N^7)$ and $O(N^4)$ of \cref{cor:d2-complexity}.
The second is a partial translation in the second control coordinate.  It leaves
those exponents unchanged and halves the leading constant in the storage count.
We first determine the support of a kernel row exactly, which refines the
bound of \cref{lem:support}.

\begin{lemma}[Exact two-threshold support]\label{lem:d2-exact-support}
For every $\ell\geq1$,
\begin{equation}\label{eq:support-d2}
 \supp K_\ell((p,q),\mathord\cdot)
 =\{(p,q)\}\cup
 \{(u,v):0\leq u\leq p,\ v\geq0,\ u+v\leq p+q-1\}.
\end{equation}
\end{lemma}

\begin{proof}
A base move never increases the first coordinate and strictly lowers the
coordinate sum, which proves the inclusion.  The terminal control $(p,q)$ is
reached by using no base move and reading endpoints until the head is
exhausted.  To reach a control $(u,v)$ in the second set with $u<p$, first
make the early-band move with $h=u$, which gives $(u,Q)$ with
$Q=p+q-1-u$.  The support inequality says $v\leq Q$.  If $v=Q$, make no
further base move.  If $v<Q$, make the last-band move with $h=v$.  If
$u=p$, then $v\leq q-1$, and the last-band move with $h=v$ can be made
directly.  Finally, exhaust the active head by endpoint choices, which
exposes the marked suffix at $(u,v)$.
\end{proof}

We now prove the first translation.

\begin{lemma}[First-coordinate translation]\label{lem:d2-translation}
For $\ell\geq1$ and $0\leq c\leq p$,
\begin{equation}\label{eq:d2-translation}
 K_\ell((p,q),(c,s))=K_\ell((p-c,q),(0,s)).
\end{equation}
\end{lemma}

\begin{proof}
No transition increases the first control coordinate: an early-band move
with parameter $h$ lowers it from $p'$ to $h$, and the other moves keep it.
Every state of a stopped path from $((p,q),(\ell)\mid L)$ to
$((c,s),\varnothing\mid L)$ therefore has first coordinate at least $c$, and
every early-band move on it has parameter $h\geq c$.  Lower every first
coordinate and every early-band parameter on the path by $c$.  By
\eqref{eq:T}, an early-band move from $(p',q')$ with parameter $h$ becomes
the move from $(p'-c,q')$ with parameter $h-c$, whose target
$(h-c,q'+p'-1-h)$ is the old target with its first coordinate lowered by
$c$.  By \eqref{eq:U}, a last-band move keeps its parameter, the second
coordinate of its target and its head enlargement, and active-head moves do
not involve the control.  The result is a stopped path from
$((p-c,q),(\ell)\mid L)$ to $((0,s),\varnothing\mid L)$, and raising every
first coordinate and every early-band parameter by $c$ is the inverse map.
\end{proof}

The proof uses only two facts: no transition increases the first control
coordinate, and, by \eqref{eq:T} and \eqref{eq:U}, if the first coordinate
of the source of a move is lowered by $c$, and its parameter $h$ also when
the move reads a value of $B_0$, then the target changes only in its first
coordinate, which is lowered by $c$, and the head enlargement is unchanged.
Both facts hold for every $d$, so that
\[
 K_\ell(\mathbf p,\mathbf t)
 =K_\ell(\mathbf p-t_0\mathbf e_0,\mathbf t-t_0\mathbf e_0)
 \qquad\hbox{whenever }t_0\leq p_0,
\]
where $\mathbf e_0=(1,0,\ldots,0)$.  Storing only the entries with $t_0=0$
and counting as in the proofs of \cref{thm:main,cor:d2-complexity} gives
$O_d(N^{3d+1})$ operations and $O_d(N^{2d})$ stored integers, that is,
$O(N^4)$ and $O(N^2)$ for $1342$, and the bounds of \cref{cor:d2-complexity}
for $12453$.  We have stated \cref{thm:main} in the simpler form.

We write
\begin{equation}\label{eq:R-definition}
 R_{\ell,a}(q,s):=K_\ell((a,q),(0,s)).
\end{equation}
We call $R_{\ell,a}(q,s)$ a \emph{reduced entry} and the table of all such
entries the \emph{first-coordinate quotient}.  Throughout this section the
letter $a$ denotes a reduced first control coordinate, not a block size as in
\eqref{eq:W} and \eqref{eq:K}.  Block sizes in a split are written
$\ell_1,\ell_2$.  By \cref{lem:d2-translation}, every kernel entry is
represented by one such reduced entry, with $a=p-c$ (entries with $c>p$
vanish by \cref{lem:d2-exact-support}).  By \cref{lem:d2-exact-support}, for a reduced source
$(\ell,a,q)$ with $a=0$ the possible terminal coordinates are
$0\leq s\leq q$, and if $a>0$ a base move is necessary and
$0\leq s\leq a+q-1$.  The first-coordinate quotient therefore has
\begin{equation}\label{eq:d2-first-reduced-counts}
 \binom{N+2}{3}\quad\hbox{rows},
 \qquad
 \frac{N(N+1)(N^2+N+4)}{12}\quad\hbox{entries}.
\end{equation}
The entry formula follows by summing $q+1$ for $a=0$ and $a+q$ for $a>0$
over the same padded source region.  In particular, it already gives
$O(N^4)$ storage.  We count the operations in the proof of
\cref{cor:d2-complexity}.

We now rewrite the kernel recurrence in the reduced variables by applying
\cref{lem:d2-translation} to \eqref{eq:K}.  An early-band move from $(a,q)$
gives $R_{\ell,h}(a+q-h-1,s)$, and a last-band move with parameter $r$,
written $r$ here because $h$ indexes the early-band sum, gives
$R_{\ell+q-r-1,a}(r,s)$.  In a split with intermediate control $(c,m)$,
the first factor $K_{\ell_1}((a,q),(c,m))$ vanishes unless $c\leq a$, since
both parts of the support \eqref{eq:support-d2} have first coordinate at
most $a$.  For $0\leq c\leq a$
the two factors become
\[
 R_{\ell_1,a-c}(q,m)R_{\ell_2,c}(m,s).
\]
The split terms of \eqref{eq:K} with a factor $K_0$ contribute
$\one_{\mathbf p=\mathbf t}$ when $\ell=1$ and $2K_{\ell-1}(\mathbf p,\mathbf t)$
when $\ell\geq2$, and we collect them into a term $D_{\ell,a}$.  Writing
$a_1=a-c$ and $a_2=c$ therefore gives the reduced recurrence
\begin{align}
R_{\ell,a}(q,s)={}&
 \sum_{h=0}^{a-1}R_{\ell,h}(a+q-h-1,s)
 +\sum_{r=0}^{q-1}R_{\ell+q-r-1,a}(r,s)
 +D_{\ell,a}(q,s)                                      \label{eq:R-recurrence}\\
&+\sum_{\substack{\ell_1,\ell_2\geq1\\
                    \ell_1+\ell_2=\ell-1}}
  \ \sum_{\substack{a_1,a_2\geq0\\a_1+a_2=a}}\ \sum_{m\geq0}
  R_{\ell_1,a_1}(q,m)R_{\ell_2,a_2}(m,s),             \notag
\end{align}
where
\[
 D_{1,a}(q,s)=\one_{a=0,\,q=s},
 \qquad
 D_{\ell,a}(q,s)=2R_{\ell-1,a}(q,s)\quad(\ell\geq2),
\]
and entries outside the support above are zero.

We call $\ell+a+q$ the \emph{reduced grade} of $R_{\ell,a}(q,\mathord\cdot)$.
Increasing reduced grade is a topological order for
\eqref{eq:R-recurrence}.  Every term of the first two sums and every $D$
term has reduced grade $\ell+a+q-1$.  In a split, the first factor has
reduced grade $\ell_1+a_1+q\leq\ell+a+q-2$.  A nonzero first factor forces
$m\leq q$ when $a_1=0$ and $m\leq a_1+q-1$ when $a_1>0$, so the second
factor has reduced grade $\ell_2+a_2+m\leq\ell+a+q-2$.

The second control coordinate admits a partial quotient as well: raising
both the second source coordinate and the terminal coordinate by one leaves
a reduced entry unchanged whenever the terminal coordinate is at least
$a-1$.

\begin{lemma}[Second-coordinate translation]\label{lem:d2-second-translation}
For $\ell\geq1$ and $a,q,s\geq0$, whenever $s\geq a-1$,
\begin{equation}\label{eq:d2-second-translation}
 R_{\ell,a}(q+1,s+1)=R_{\ell,a}(q,s).
\end{equation}
\end{lemma}

\begin{proof}
We label the unread values.  In a state with stack $U\mid L$, give the
values of band~$0$, of band~$1$, and of the intervals of $U$ distinct
labels, ordered as the values are.  The transitions of
\eqref{eq:H} from such a state correspond to the labels of the two bands and
of the active head.  An early-band move with parameter $h$ reads the label
of band~$0$ with $h$ labels of band~$0$ below it, and the labels of
band~$0$ above it join band~$1$, below its labels.  A last-band move with
parameter $h$ reads the label of band~$1$ with $h$ labels of band~$1$ below
it, and the labels of band~$1$ above it join the active head, below its
labels.  A letter of the active head of local rank $r$ reads the head
label with $r-1$ head labels below it, and the head is replaced by its labels
below and above that label, in this order.  So a stopped path from a
labeled state is determined by the sequence of labels it reads.

Consider a stopped path counted by $R_{\ell,a}(q+1,s+1)$, and let $z$ be the
least label of the initial band~$1$.  Suppose that $z$ leaves band~$1$, by
being read or by joining the head.  This happens in a last-band move that
reads $z$ or a label below it, and every label of the initial band~$1$ still
in band~$1$ joins the head or is read in the same move.  Afterwards labels
enter band~$1$ only from band~$0$, so the $s+1$ labels of the terminal
band~$1$ belong to the initial band~$0$.  Band~$0$ loses labels only through
early-band moves, each of which reads one of its labels, and it ends empty.
So at most $\max(a-1,0)$ labels of the initial band~$0$ remain unread, while
$s+1\geq\max(a,1)$ since $s\geq a-1$.  This contradiction shows that $z$
stays unread in band~$1$.  In particular every last-band move reads a label
above $z$.  Deleting $z$ from every state leaves the early-band parameters
and the local ranks of head letters unchanged, and lowers each last-band
parameter and the
size of band~$1$ by one, which leaves each head enlargement $\delta_h$
unchanged.  Since $z$ never enters the stack, the stacks are unchanged and
$L$ is exposed at the same step.  This gives a stopped path counted by
$R_{\ell,a}(q,s)$ that reads the same labels.

Conversely, consider a stopped path counted by $R_{\ell,a}(q,s)$.  No
last-band move of it reads a label $y$ of the initial band~$0$.  Such a $y$
would have entered band~$1$ through an early-band move, which read another
label of the initial band~$0$, and reading $y$ would send every label of the
initial band~$1$ still in band~$1$ into the head.  The terminal band~$1$ would then consist of
at most $a-2$ labels of the initial band~$0$, contradicting $s\geq a-1$.  So
every last-band move reads a label of the initial band~$1$.  Insert a label
$z$ into the initial band~$1$ below its labels.  The same labels can then be
read, with each last-band parameter and the size of band~$1$ raised by one
and the head enlargements unchanged.  Since every last-band move reads a
label above $z$, the label $z$ stays unread in band~$1$, so the stacks are
unchanged and $L$ is exposed at the same step.  This gives a stopped path
counted by $R_{\ell,a}(q+1,s+1)$.  The two operations
are inverse.
\end{proof}

The hypothesis $s\geq a-1$ cannot be weakened: equation
\eqref{eq:R-recurrence} gives $R_{1,2}(1,1)=3$ but $R_{1,2}(0,0)=2$.

For $a=0$ and $0\leq s\leq q$, \cref{lem:d2-second-translation} gives
$R_{\ell,0}(q,s)=R_{\ell,0}(q-s,0)$.  For $a\geq1$ and $s\geq a$, it gives
\[
 R_{\ell,a}(q,s)=R_{\ell,a}(q-s+a-1,a-1),
\]
where the support condition ensures that the new second source coordinate is
nonnegative.  Thus we retain only terminal coordinate $s=0$ when $a=0$, and
terminal coordinates
$0\leq s<a$ when $a\geq1$.  The exact number of retained entries is
\begin{align}
 \binom{N+1}{2}+\binom{N+2}{4}
 &=\frac{N(N+1)(N^2+N+10)}{24}.             \label{eq:d2-reduced-counts}
\end{align}
Thus the second translation halves the leading storage constant relative to
the first-coordinate quotient, and each omitted entry is recovered from
\eqref{eq:d2-second-translation} in $O(1)$ operations.  Exact polynomials for loop iterations and
stored entries are recorded in the computation report
\path{notes/av12453_150_computation_report.md} of the supplement.

\begin{corollary}[Sharpened $12453$ bound]\label{cor:d2-complexity}
For every $N\geq1$, the numbers $\lvert\Av_n(12453)\rvert$ for $0\leq n\leq N$
can be computed using
$O(N^7)$ integer arithmetic operations and $O(N^4)$ stored integers.
The bit complexity is $O(N^7M(N\log N))$, and the bit storage is
$O(N^5\log N)$, with $M$ as in \cref{thm:main}.
\end{corollary}

\begin{proof}
The algorithm evaluates \eqref{eq:R-recurrence} for the reduced sources
$(\ell,a,q)$ with $\ell+a+q\leq N$ in increasing reduced grade, which the
preceding paragraphs showed to be a topological order, then evaluates
\eqref{eq:G}, reading every kernel entry from the first-coordinate quotient
by \eqref{eq:d2-translation}, and returns the values \eqref{eq:answer-G}.
It is correct by the proof of \cref{thm:main} and
\cref{lem:d2-exact-support,lem:d2-translation}.  We count its operations.
The first-coordinate quotient has
$O(N^3)$ source rows and $O(N^4)$ stored entries by
\eqref{eq:d2-first-reduced-counts}.  Evaluating the split terms involves $O(N)$ choices for
the block split $\ell_1$, $O(N)$ choices for the control split $a_1$, and
$O(N)$ possible intermediate second coordinates $m$.  For each such triple,
adding the row $R_{\ell_2,a_2}(m,\mathord\cdot)$ uses $O(N)$ possible terminal
coordinates.  Thus one source row costs $O(N^4)$ operations and all split
terms cost $O(N^7)$.  Evaluating
the nonsplit kernel terms costs at most $O(N^5)$.  The empty-stack table has
$O(N^2)$ source controls, $O(N)$ last-band choices, and $O(N^2)$ possible
terminal controls, so it also costs $O(N^5)$.  The bit-size argument
in the proof of \cref{thm:main} gives $O(N\log N)$ bits per integer, which
gives the stated bit bounds.
\end{proof}

\section{The counting sequence of \texorpdfstring{$\Av(12453)$}{Av(12453)}}
\label{sec:computations}

In this section we describe the computation of $\lvert\Av_n(12453)\rvert$ for
$n\leq150$ and, in \cref{sec:sampling}, uniform random generation.
\Cref{tab:terms-150} lists some of the values.  The file
\path{code/data/av12453_terms_0_150.txt} in the supplement contains all of
them.

\begin{table}[tb]
\centering
\scriptsize
\begin{tabular}{@{}r>{\raggedright\arraybackslash}p{0.88\textwidth}@{}}
\toprule
$n$ & $\lvert\Av_n(12453)\rvert$ \\
\midrule
$ 0 $ & $ 1 $ \\
$ 1 $ & $ 1 $ \\
$ 2 $ & $ 2 $ \\
$ 3 $ & $ 6 $ \\
$ 4 $ & $ 24 $ \\
$ 5 $ & $ 119 $ \\
$ 6 $ & $ 694 $ \\
$ 7 $ & $ 4581 $ \\
$ 8 $ & $ 33286 $ \\
$ 9 $ & $ 260927 $ \\
$\vdots$ & $\vdots$ \\
$ 70 $ & $ 2720470144\allowbreak 2650584805\allowbreak 3184970016\allowbreak 4950418633\allowbreak 1718343803\allowbreak 4114760726\allowbreak 0652518243\allowbreak 22 $ \\
$ 71 $ & $ 3650982161\allowbreak 8390692452\allowbreak 4470453450\allowbreak 9742637654\allowbreak 0767263491\allowbreak 5822649288\allowbreak 4931736173\allowbreak 219 $ \\
$ 72 $ & $ 4905247912\allowbreak 2475377086\allowbreak 4377575385\allowbreak 4964333946\allowbreak 1005664517\allowbreak 0350477388\allowbreak 6106466065\allowbreak 0462 $ \\
$ 73 $ & $ 6597581300\allowbreak 4526715227\allowbreak 9725187963\allowbreak 5194170286\allowbreak 5546101013\allowbreak 4067851632\allowbreak 9598742991\allowbreak 39314 $ \\
$ 74 $ & $ 8883183490\allowbreak 3029995386\allowbreak 3181157479\allowbreak 2997430948\allowbreak 6984701302\allowbreak 4502094615\allowbreak 2645591389\allowbreak 765992 $ \\
$ 75 $ & $ 1197293585\allowbreak 2803456329\allowbreak 9140115867\allowbreak 5897325840\allowbreak 3718404449\allowbreak 7744461031\allowbreak 1323349232\allowbreak 27197902 $ \\
$ 76 $ & $ 1615360029\allowbreak 7873711094\allowbreak 1515106002\allowbreak 9379769950\allowbreak 5865894707\allowbreak 0938911533\allowbreak 7666950300\allowbreak 899664892 $ \\
$ 77 $ & $ 2181542206\allowbreak 2572671915\allowbreak 6139395376\allowbreak 9585550120\allowbreak 9448778205\allowbreak 4427029330\allowbreak 1176937761\allowbreak 8433242686 $ \\
$ 78 $ & $ 2948987237\allowbreak 9065039096\allowbreak 2604916545\allowbreak 1214495399\allowbreak 5161437891\allowbreak 5249501191\allowbreak 0397390414\allowbreak 8385295980\allowbreak 0 $ \\
$ 79 $ & $ 3990128916\allowbreak 9830238253\allowbreak 4404743222\allowbreak 7024955428\allowbreak 5379834607\allowbreak 8734627168\allowbreak 6479611600\allowbreak 9946452966\allowbreak 10 $ \\
$ 80 $ & $ 5403757815\allowbreak 1511512886\allowbreak 6978959334\allowbreak 2973010150\allowbreak 6005764295\allowbreak 7609146014\allowbreak 4034839052\allowbreak 7553113513\allowbreak 544 $ \\
$\vdots$ & $\vdots$ \\
$ 140 $ & $ 1325631292\allowbreak 0967890009\allowbreak 7059991729\allowbreak 9726854851\allowbreak 6229398560\allowbreak 9301526390\allowbreak 6399584352\allowbreak 8477800927\allowbreak 8831298106\allowbreak 9030946201\allowbreak 8830746612\allowbreak 7675440342\allowbreak 4995491130\allowbreak 4218723154\allowbreak 3841227010\allowbreak 96 $ \\
$ 141 $ & $ 1852883512\allowbreak 6015863666\allowbreak 5154188074\allowbreak 8812997339\allowbreak 0850367330\allowbreak 1584716105\allowbreak 5671375191\allowbreak 6345315609\allowbreak 2505084472\allowbreak 6568844410\allowbreak 4442573164\allowbreak 2146730167\allowbreak 0647930415\allowbreak 7265601031\allowbreak 7336745260\allowbreak 230 $ \\
$ 142 $ & $ 2590615022\allowbreak 2930343813\allowbreak 7253209397\allowbreak 7478210955\allowbreak 6708042335\allowbreak 7798067579\allowbreak 8231493541\allowbreak 9893631138\allowbreak 4115189454\allowbreak 8670967420\allowbreak 1933295879\allowbreak 4604234783\allowbreak 1801662279\allowbreak 7071206729\allowbreak 5551728877\allowbreak 3276 $ \\
$ 143 $ & $ 3623141475\allowbreak 0966600876\allowbreak 9192778238\allowbreak 0212981344\allowbreak 7986225375\allowbreak 2577900601\allowbreak 6335293373\allowbreak 9873961213\allowbreak 2874209755\allowbreak 0042572534\allowbreak 3682999003\allowbreak 8858913060\allowbreak 9418876136\allowbreak 3063724056\allowbreak 3247858231\allowbreak 77873 $ \\
$ 144 $ & $ 5068665701\allowbreak 0808099217\allowbreak 2827702390\allowbreak 8247214853\allowbreak 1459218423\allowbreak 6068629394\allowbreak 5507096335\allowbreak 7278450755\allowbreak 8031741140\allowbreak 3139817900\allowbreak 5469949048\allowbreak 0890332852\allowbreak 4631685357\allowbreak 7908221730\allowbreak 0140047632\allowbreak 872622 $ \\
$ 145 $ & $ 7092939827\allowbreak 6475491959\allowbreak 3192205776\allowbreak 2445045917\allowbreak 2585455841\allowbreak 0034064407\allowbreak 2579925164\allowbreak 5016654364\allowbreak 5900069113\allowbreak 8336169565\allowbreak 0825722175\allowbreak 6849768300\allowbreak 9594526954\allowbreak 0900306222\allowbreak 2621709935\allowbreak 1149132 $ \\
$ 146 $ & $ 9928451859\allowbreak 5778899542\allowbreak 6922964104\allowbreak 6456946313\allowbreak 5271052892\allowbreak 3383764487\allowbreak 5179969638\allowbreak 5456954495\allowbreak 8714672498\allowbreak 7087004389\allowbreak 9715835749\allowbreak 1680840589\allowbreak 6796223500\allowbreak 9591429889\allowbreak 6248098119\allowbreak 96930400 $ \\
$ 147 $ & $ 1390137692\allowbreak 8245085643\allowbreak 3777786818\allowbreak 0587785468\allowbreak 8279319447\allowbreak 6967952737\allowbreak 6580482443\allowbreak 6580175611\allowbreak 0716322666\allowbreak 2022575277\allowbreak 0817613657\allowbreak 7674974169\allowbreak 9180828349\allowbreak 4959301849\allowbreak 1511764347\allowbreak 3649518426 $ \\
$ 148 $ & $ 1946944429\allowbreak 4891725939\allowbreak 1570834796\allowbreak 9444206526\allowbreak 7747652257\allowbreak 8329510621\allowbreak 5718353174\allowbreak 0388546951\allowbreak 9572114751\allowbreak 3596204003\allowbreak 2085429461\allowbreak 2592322914\allowbreak 3342885756\allowbreak 7215495002\allowbreak 6956201648\allowbreak 1062878067\allowbreak 2 $ \\
$ 149 $ & $ 2727515338\allowbreak 5986397634\allowbreak 7598705423\allowbreak 9279554827\allowbreak 9852776796\allowbreak 3594985771\allowbreak 5137514148\allowbreak 6565870266\allowbreak 8306436986\allowbreak 9530762605\allowbreak 6547310084\allowbreak 4770203223\allowbreak 1964206135\allowbreak 1163298825\allowbreak 1274688398\allowbreak 6281365819\allowbreak 35 $ \\
$ 150 $ & $ 3822057739\allowbreak 1784040990\allowbreak 0771979163\allowbreak 2020545156\allowbreak 2668309691\allowbreak 7522113147\allowbreak 0523342088\allowbreak 8592406553\allowbreak 2828958769\allowbreak 7088297741\allowbreak 5112039607\allowbreak 5850075847\allowbreak 7699391695\allowbreak 5095899208\allowbreak 9859207315\allowbreak 9171767612\allowbreak 574 $ \\
\bottomrule
\end{tabular}
\caption{The numbers $\lvert\Av_n(12453)\rvert$ for $0\leq n\leq9$, $70\leq n\leq80$,
and $140\leq n\leq150$.}
\label{tab:terms-150}
\end{table}

We evaluated the reduced recurrence modulo pairwise distinct primes and
reconstructed the values over the integers.  The programs
\path{av12453_fast_rns.cpp} and \path{reconstruct_av12453_rns.py} in the
supplement perform the modular evaluation, the check of the coefficient
bound, and the reconstruction.

The reconstruction rests on a bound for the coefficients.  With
$a_m^{(1)}=\lvert\Av_m(1342)\rvert$ as in \eqref{eq:initial-terminal}, put
\begin{equation}\label{eq:exact-crt-bound}
 \mathcal B_N:=\sum_{m=0}^N\binom Nm^2 a_m^{(1)}.
\end{equation}
As B\'ona~\cite{Bona2005} observed, deleting the left-to-right minima of a
$12453$-avoider leaves a $1342$-avoider: the first entry of an occurrence of
$1342$ among the remaining entries is not a left-to-right minimum, so an
earlier smaller entry precedes it and forms a $12453$ with the occurrence.
If $m$ entries remain, choose their positions, their values, and their
standardization.  The deleted minima occupy the remaining positions and take
the remaining values in decreasing order.  Therefore
$\lvert\Av_N(12453)\rvert\leq\mathcal B_N$.

The integers $a_m^{(1)}$ are computed from \eqref{eq:bona-1342}.  The
source \path{av12453_fast_rns.cpp} implements \eqref{eq:R-recurrence} and
the empty-stack recurrence \eqref{eq:G} modulo each selected prime.  By
\cref{thm:literal,cor:protected-tail,prop:kernel-recurrence,%
cor:empty-stack-recurrence,lem:d2-exact-support,lem:d2-translation,%
lem:d2-second-translation}, these recurrences give the coefficients, so a
correct evaluation outputs their residues, and the checks below test the
implementation.  The moduli are the
twenty-four largest primes below $2^{31}$, recorded with the residue packs
in the supplement.  The program \path{reconstruct_av12453_rns.py} verifies
that the product $\mathcal M$ of the eighteen largest exceeds
$\mathcal B_{150}$.  The margin is $\mathcal M/\mathcal B_{150}>1.92$.  Since
the nonnegative bound $\mathcal B_N$
is nondecreasing, every coefficient through degree $150$ lies in
$[0,\mathcal M)$.  The Chinese remainder theorem therefore determines each
coefficient uniquely.

The six remaining primes, withheld from the reconstruction, agree with every
reconstructed coefficient.  A separately written single-modulus $d=2$
program computes all $151$ coefficients modulo $2^{61}-1$ and also agrees.
These are error-detection checks, not inputs to the uniqueness argument.
The values through $n=38$ agree with those previously known, in OEIS
A116485~\cite{OEISA116485}, which come from Biers-Ariel's
program~\cite{BiersArielA116485}.

\subsection{Uniform random generation}\label{sec:sampling}

The transfer tables can also be used to generate uniformly random avoiders.  We
call a maximal transition sequence a \emph{complete path}.  By
\cref{thm:literal}, the complete paths from the initial state
$((n,0,\ldots,0),\varnothing)$ correspond bijectively to the
$\beta_d$-avoiding permutations of length $n$.  By the recursive method of
Nijenhuis and Wilf~\cite{NijenhuisWilf1978}, such a path is uniformly random
when every transition is chosen with probability proportional to the number
of complete paths through it, that is, to the value of the
corresponding summand of \eqref{eq:H}.  That value is a completion count
$H_{\mathbf p}(L)$ for a stack $L$ of arbitrary length, which the
factorization \eqref{eq:matrix-product} expresses through the stored
kernels.  We therefore sample a complete path of the factorized recurrence
itself.  The first-move partitions in the proofs of
\cref{prop:kernel-recurrence,cor:empty-stack-recurrence} identify the
summands of \eqref{eq:K} and \eqref{eq:G} with a partition of the stopped
and complete paths.  Iterating the cut-and-concatenate bijection of
\cref{thm:protected-tail-principle} maps these paths onto the complete
paths of \eqref{eq:H}.  At an empty-stack state, a summand of \eqref{eq:G}
is chosen with probability proportional to its value.  This selects a base
letter and, for a last-band move, the terminal control $\mathbf t$ at which
the new head will have been exhausted.  Inside a kernel
$K_\ell(\mathbf p,\mathbf t)$, a summand of \eqref{eq:K} is chosen in the
same way.  This selects a base move, an endpoint of the active head, or an
interior letter of it together with the intermediate control $\mathbf u$ of
the split, and each block of the stack carries the control at which it must be
exposed.  By induction on the grade and the control mass, a given stopped path
counted by $K_\ell(\mathbf p,\mathbf t)$ is produced with probability
$1/K_\ell(\mathbf p,\mathbf t)$, and a given complete path from the state
$(\mathbf p,\varnothing)$ with probability $1/G_{\mathbf p}$.  At
$\mathbf p=(n,0,\ldots,0)$ this is $1/\lvert\Av_n(\beta_d)\rvert$.

\begin{proposition}[Uniform random generation]\label{prop:sampling}
Let $N\geq n\geq0$.  Once the first-coordinate quotient $R_{\ell,a}(q,s)$ of
\cref{sec:d2-support} and the empty-stack values $G_{(p,q)}$ have been
computed for size $N$, a uniformly random element of $\Av_n(12453)$ can be
produced with $O(n^4)$ arithmetic operations on the stored integers.
\end{proposition}

\begin{proof}
The paragraph before the proposition shows that the distribution is uniform.  A permutation of length $n$ is
produced by $n$ letter choices.  Each choice scans the summands of one
instance of \eqref{eq:G} or \eqref{eq:K} in a fixed order and stops at the first summand whose running
total exceeds a uniformly random integer $y$ with $0\leq y<Y$, where $Y$ is
the stored value of the state.  Every summand is one stored entry or a
product of two, by \cref{lem:d2-translation}.  Every kernel row consulted
has reduced grade at most $n$, and every empty-stack control has mass at
most $n$.  The sampler keeps the value sets of the bands and of the
intervals of the stack as sorted lists, so that it outputs the letter of
each chosen transition and updates the sets in $O(n)$ operations.  At an empty-stack state the
last-band terms range over the band position and the terminal control,
$O(n^3)$ summands.  Inside a kernel the interior terms range over the split
position and the intermediate control, again $O(n^3)$.  The other groups are
smaller.  Therefore each choice costs $O(n^3)$ operations and a sample costs
$O(n^4)$.
\end{proof}

Unlike the coefficients of \cref{tab:terms-150}, our sampler runs in
floating point.  \Cref{prop:sampling} is the exact-arithmetic statement.
\Cref{app:float} describes the implementation and proves that, for
$n\leq300$ and with ideal random bits, its output is within total variation
distance
$3.5\cdot10^{-5}$ of the uniform distribution on $\Av_n(12453)$
(\cref{prop:float-sampling}).  Every sample is verified to avoid $12453$ by
an independent test based on \cref{lem:trigger}, and the sampler was checked
against exhaustive enumeration for $7\leq n\leq10$.  Building
the tables for $N=300$ took about four hours on an eight-core aarch64
machine and produced a table of $5.4$~GB, which the sampler holds in
memory.  One million
samples of length $300$ then took about thirty-five minutes.

\Cref{fig:heatmap} shows the resulting position--value heatmap, rendered
with the script that draws the heatmaps of the PermPAL database
\cite{PermPAL,AlbertEtAl2026}, beside one of the samples.  The first letter
has mean $291.9$ over the one million samples.  The dark band from the top left to the bottom
right shows that a typical avoider stays close to a decreasing curve.  The
left-to-right minima lie below it, and their mean number is $77.5$.  The single
sample shows the same features: a decreasing band of letters with the
left-to-right minima scattered below it, and above it the letters
exceeding the first letter, which form a $1342$-avoider.

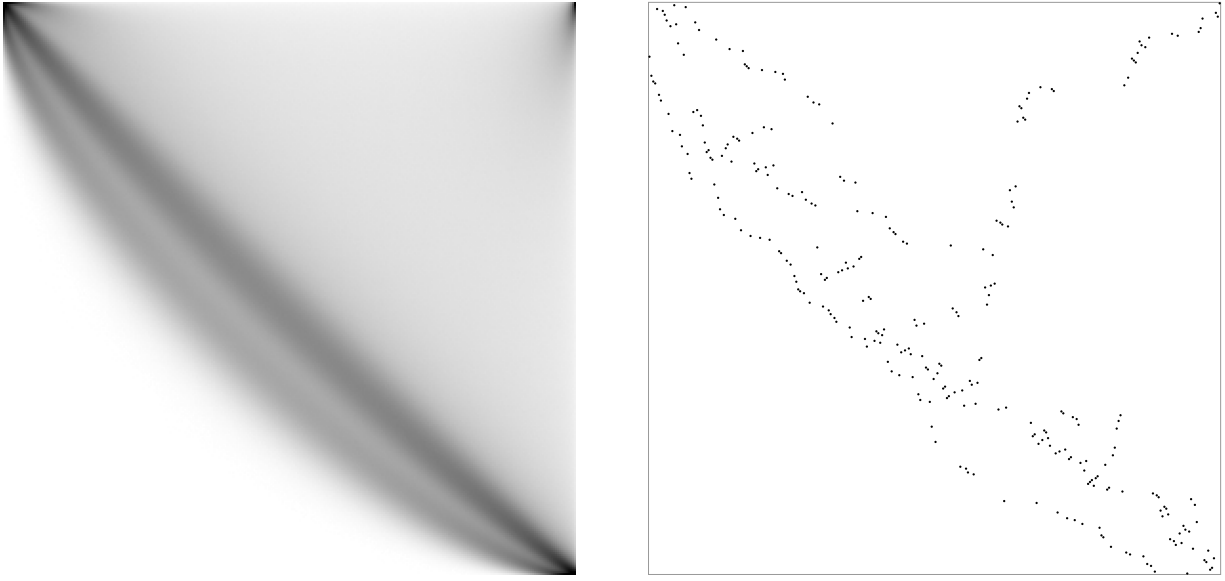
\begin{figure}[tb]
\centering
\begin{minipage}[c]{0.47\textwidth}\centering
\includegraphics[width=\textwidth]{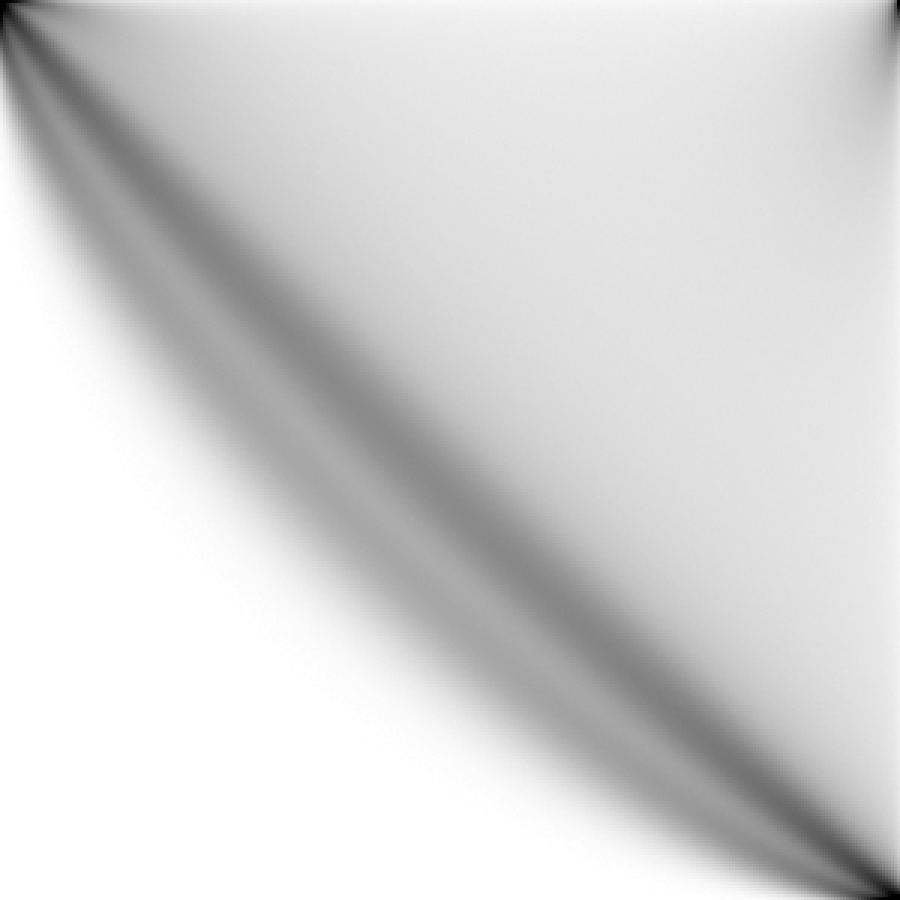}
\end{minipage}\hfill
\begin{minipage}[c]{0.47\textwidth}\centering
\setlength{\permcell}{\textwidth}\divide\permcell by 300\relax
\setlength{\permdot}{3\permcell}\divide\permdot by 5\relax 
\begin{tikzpicture}[x=\permcell,y=\permcell]
  \useasboundingbox (0.5,0.5) rectangle (300.5,300.5);
  \draw[black!40,line width=0.3pt] (0.5,0.5) rectangle (300.5,300.5);
  \foreach \i/\v in {%
    1/272, 2/262, 3/259, 4/258, 5/297, 6/252, 7/249, 8/296, 9/294, 10/291,
    11/242, 12/288, 13/233, 14/299, 15/289, 16/279, 17/231, 18/225, 19/273,
    20/298, 21/221, 22/211, 23/208, 24/243, 25/290, 26/244, 27/286, 28/241,
    29/236, 30/227, 31/222, 32/223, 33/219, 34/218, 35/205, 36/281, 37/198,
    38/192, 39/220, 40/189, 41/224, 42/226, 43/276, 44/217, 45/230, 46/187,
    47/229, 48/228, 49/181, 50/275, 51/268, 52/267, 53/266, 54/178, 55/232,
    56/216, 57/212, 58/213, 59/177, 60/265, 61/235, 62/214, 63/210, 64/176,
    65/234, 66/215, 67/264, 68/203, 69/170, 70/169, 71/263, 72/260, 73/165,
    74/200, 75/163, 76/199, 77/157, 78/154, 79/150, 80/149, 81/201, 82/148,
    83/197, 84/251, 85/143, 86/195, 87/248, 88/194, 89/172, 90/247, 91/158,
    92/141, 93/155, 94/156, 95/139, 96/137, 97/237, 98/135, 99/133, 100/159,
    101/209, 102/160, 103/207, 104/164, 105/161, 106/130, 107/125, 108/162,
    109/206, 110/191, 111/166, 112/167, 113/144, 114/124, 115/120, 116/146,
    117/145, 118/190, 119/123, 120/128, 121/127, 122/122, 123/126, 124/129,
    125/188, 126/112, 127/182, 128/107, 129/180, 130/179, 131/121, 132/105,
    133/117, 134/175, 135/118, 136/174, 137/119, 138/116, 139/104, 140/134,
    141/131, 142/95, 143/92, 144/115, 145/132, 146/109, 147/108, 148/91,
    149/78, 150/103, 151/70, 152/106, 153/111, 154/110, 155/98, 156/99,
    157/93, 158/94, 159/173, 160/140, 161/96, 162/138, 163/136, 164/57,
    165/97, 166/89, 167/56, 168/54, 169/102, 170/100, 171/53, 172/90,
    173/101, 174/113, 175/114, 176/171, 177/151, 178/142, 179/147, 180/152,
    181/168, 182/153, 183/186, 184/87, 185/185, 186/184, 187/39, 188/88,
    189/183, 190/202, 191/196, 192/193, 193/204, 194/238, 195/246, 196/245,
    197/240, 198/239, 199/250, 200/253, 201/80, 202/73, 203/74, 204/38,
    205/69, 206/256, 207/71, 208/76, 209/75, 210/72, 211/68, 212/255,
    213/254, 214/64, 215/33, 216/65, 217/86, 218/85, 219/66, 220/30, 221/61,
    222/62, 223/83, 224/29, 225/82, 226/79, 227/59, 228/27, 229/55, 230/60,
    231/48, 232/49, 233/50, 234/47, 235/51, 236/52, 237/25, 238/21, 239/20,
    240/58, 241/45, 242/46, 243/15, 244/63, 245/67, 246/77, 247/81, 248/84,
    249/44, 250/257, 251/12, 252/261, 253/11, 254/271, 255/270, 256/269,
    257/274, 258/280, 259/278, 260/10, 261/277, 262/6, 263/282, 264/5,
    265/43, 266/2, 267/42, 268/41, 269/34, 270/31, 271/36, 272/35, 273/32,
    274/19, 275/284, 276/18, 277/16, 278/283, 279/22, 280/17, 281/26,
    282/24, 283/1, 284/23, 285/40, 286/14, 287/37, 288/28, 289/285, 290/287,
    291/292, 292/8, 293/7, 294/13, 295/3, 296/4, 297/9, 298/295, 299/293,
    300/300}
    \fill (\i,\v) circle[radius=\permdot];
\end{tikzpicture}%
%
\end{minipage}
\caption{Left: position--value heatmap of one million elements of
$\Av_{300}(12453)$ drawn by the floating-point sampler.  The cell at position $i$ (left to right) and value $v$
(bottom to top) is shaded by the square root of the number of samples with
$\pi_i=v$, from white (no sample) to black (the largest cell count,
$70\,204$).  Right: one of the samples, with a dot at $(i,\pi_i)$ in the
same orientation.  Its first letter is $272$ and it has $71$ left-to-right
minima.}
\label{fig:heatmap}
\end{figure}

\section{Formal verification}\label{sec:formal}

The literal and kernel recurrences of this paper for $d\leq2$ have been
verified in the Lean~4 proof assistant \cite{Lean4} with the Mathlib library
\cite{Mathlib}.  The development is in the directory \path{formal/} of the
supplement.  Its four main theorems state that, for every $n$,
\begin{align*}
 W_n(\varnothing)&=\lvert\Av_n(1342)\rvert,
 &G_n&=\lvert\Av_n(1342)\rvert,\\
 H_{(n,0)}(\varnothing)&=\lvert\Av_n(12453)\rvert,
 &G_{(n,0)}&=\lvert\Av_n(12453)\rvert.
\end{align*}
The development proves the counting statement of \cref{thm:literal} for
$d\leq2$, that a legal prefix with state $(\mathbf p,L)$ has
$H_{\mathbf p}(L)$ completions to a $\beta_d$-avoiding permutation, and the
first and third theorems are its case of the empty prefix.  The second and
fourth say that the recurrences evaluated by the algorithm of
\cref{sec:algorithm} count $\Av_n(1342)$ and $\Av_n(12453)$.

A reader who wants to verify this needs to do two things.  The first
is to run the scripts \path{check.sh} and \path{comparator.sh}.  The script
\path{check.sh} builds the development and checks that its proof libraries
and the file \path{Solution.lean} contain no unproved statement and that
every theorem of the libraries depends only on Lean's three standard axioms.
The script \path{comparator.sh} checks, with Lean's kernel and the two
independent kernels nanoda and con-ron, that the proofs in
\path{Solution.lean} prove exactly the statements of the file
\path{Challenge.lean} and use no other axioms.  In \path{Challenge.lean} each
of the four statements has the placeholder \texttt{sorry} in place of its
proof.  The second thing to do is to read \path{Challenge.lean}, about
180 lines that import only Mathlib and contain the four statements together
with every definition they use: pattern containment, the patterns $1342$ and
$12453$, and the recurrences.  The proofs need not be read.

The proofs follow the paper.  The interval stack of a scanned prefix is
computed from the prefix by the moves of \cref{prop:state-invariant}(b) and
its control from the thresholds, the separation criterion of \cref{lem:separators} is the
inductive invariant, \cref{prop:scan-states}(a) and (b) are derived from it,
and the recurrences follow by induction on the number of unread values.  The
kernels are defined by their recurrences, and the factorization
\eqref{eq:factorization} for a single head $U=(\ell)$ is proved from the
equations of $H$ by induction on the grade rather than by cutting paths.
Not formalized are, among others, the complexity statements of
\cref{thm:main,cor:d2-complexity}, the description of the kernels as counts
of stopped paths, \cref{prop:sampling} and \cref{app:float}, condition (iii)
of \cref{cor:separators}, \cref{prop:exponential}, and the family for
$d\geq3$.  The translation quotients of \cref{sec:d2-support} and the bound
$\lvert\Av_N(12453)\rvert\leq\mathcal B_N$, on which the programs of
\cref{sec:computations} rely, are not formalized either.  \Cref{app:lean} describes the development in more
detail.

\section{Consequences and outlook}\label{sec:discussion}

In this section we place $12453$ among the length-five classes and state a
conjecture on the asymptotics of its counting sequence.  Two patterns are \emph{Wilf-equivalent} if they have equally many avoiders
of every length.  Reverse, complement, and inverse preserve all counting
sequences.  The dihedral orbit of $12453$ has eight elements, and the
prefix-reversal equivalence of Backelin--West--Xin
\cite{BackelinWestXin2007} gives
$12453=12 \oplus 231 \sim 21 \oplus 231=21453$, whose orbit supplies eight more.  Therefore the
algorithm, bounds, and coefficients apply to all sixteen members
of the length-five Wilf class represented by $12453$, which has no other
members~\cite[Table~1]{ClisbyConwayGuttmannInoue2022}.

B\'ona determined the Stanley--Wilf limit of every member of the
family \cite{Bona2005}:
\begin{equation*}
 \lim_{n\to\infty}\bigl(a_n^{(d)}\bigr)^{1/n}=(d-1+\sqrt{8})^2 .
\end{equation*}
For the remainder of this section, put
$a_n=a_n^{(2)}=\lvert\Av_n(12453)\rvert$ and $\mu=(1+\sqrt{8})^2=9+4\sqrt{2}$.
The coefficients through $n=100$, read with the series analysis of Clisby,
Conway, Guttmann, and Inoue~\cite{ClisbyConwayGuttmannInoue2022}, suggest the
following refinement.

\begin{conjecture}\label{conj:12453-asymptotic}
There are constants $C,\kappa>0$ and $g,h\in\mathbb R$ such that, as
$n\to\infty$,
\begin{equation}\label{eq:12453-asymptotic}
 a_n=C\mu^n e^{-\kappa n^{1/3}}n^{g}
 \left(1+h n^{-1/3}+O(n^{-2/3})\right).
\end{equation}
\end{conjecture}

The stretched exponent $1/3$ agrees with the value that the 2022 analysis
found most likely~\cite[Section~10.3]{ClisbyConwayGuttmannInoue2022}.
Unweighted least squares on $70\leq n\leq100$ for the model
\[
 \log a_n=n\log\mu-\kappa n^{1/3}+g\log n+\log C+h n^{-1/3}
\]
gives $g\approx-4.251$, and we fixed the nearby value $g=-17/4$.  Both
exponents were chosen using only $a_0,\ldots,a_{100}$.  With
$g=-17/4$, the same fit gave the parameters $\kappa,\log C,h$, which were
fixed before $a_{101},\ldots,a_{150}$ were computed.  The frozen values are
$\kappa=1.34550865$, $C=0.82710184$ and $h=1.36800770$, to the digits
shown.
On those fifty excluded coefficients, the maximum absolute difference
between the observed and predicted values of $\log a_n$ is below
$4.3\cdot10^{-7}$.  The two fits and this check are reproduced by a
program in the supplement.

From shorter series, the 2022 analysis concluded that the
stretched-exponential base $e^{-\kappa}$ (their $\mu_1$) is $0.27\pm0.02$
and that the power-law exponent is
$-4.7\pm0.3$~\cite[Section~10.3]{ClisbyConwayGuttmannInoue2022}.  The frozen
fit gives $e^{-\kappa}\approx0.260$, inside that range.  Its exponent
$-17/4$ lies outside the 2022 range, but we do not regard the two as in
conflict.  On $70\leq n\leq150$ the exponent $-17/4$ fits the model with
one correction term far better than their single-fit value $-4.67$ does.
The preference shrinks by more than an order of magnitude once a second
correction term $n^{-2/3}$ is admitted.  In fits with two correction terms,
a change in the stretched exponent $1/3$ is largely compensated by changes
in the power-law exponent and in $\kappa$.  The inferred power-law exponent
is therefore sensitive to the correction terms included, and the available
coefficients do not determine it reliably, which is why
\cref{conj:12453-asymptotic} leaves $g$ unspecified.  We report
$g\approx-17/4$ as the value fitted with the stretched exponent $1/3$ and
one correction term, not as a sharpening of the 2022 estimate.  The holdout
test supports \cref{conj:12453-asymptotic} but does not prove it.

\section*{Data and code availability}

The code, coefficients, residues, reconstruction script, and computation
reports used for \cref{sec:computations}, the sampler and heatmap tools of
\cref{sec:sampling}, the Permuta walkthrough with its exhaustive checks for
length at most $8$, and the Lean development of \cref{sec:formal} and \cref{app:lean} are in the
public repository
\begin{center}
\url{https://github.com/ulfarsson/public-12453}
\end{center}
\noindent
in its release v1.4.0, which is the supplement referred to throughout the
paper.  The code, data, and Lean development are released under the Apache
License 2.0.  The heatmap of \cref{fig:heatmap} was drawn from the count
matrix in the repository by the heatmap script of the PermPAL database,
written by Jay Pantone and Vince Vatter, which is not part of the
repository.  The repository's
\path{MANIFEST.sha256} records the SHA-256 digest of every file, and its README files
give verification, reconstruction, and audit commands.

\section*{Statement on the use of AI}
Large language model assistants (Claude Opus and Claude Fable from
Anthropic, and GPT Sol from OpenAI) were used throughout the preparation of
this paper under the direction of the author: to draft and
revise the text, to write and test the companion programs and run the
computations reported in \cref{sec:computations}, and to develop the Lean
formalization of \cref{sec:formal}.  The mathematical content was specified
and checked by the author.  The computational claims are supported by the
independent implementations and checks described in
\cref{sec:computations} and, for the sampler, by the error bound of
\cref{app:float}, and the statements of \cref{sec:formal} are verified
by the Lean kernel.  In addition, the Permuta walkthrough in the
supplement compares finite instances of the paper's combinatorial statements
with brute-force computations from the definitions, including exhaustive
checks on permutations of length at most $8$.  Each check records its
parameter range and scope.  These checks are written with the Permuta library
\cite{Permuta}, whose core code predates the large language model assistants, and
they do not use the implementations described in \cref{sec:computations}.
The author takes full responsibility for the paper.

\appendix
\crefalias{section}{appendix}
\section{Floating-point sampling}\label{app:float}

This appendix describes the floating-point implementation of the sampler of
\cref{prop:sampling} and bounds the distance of its output distribution from
uniformity.  Throughout, the tables have size $N=300$ and $n\leq N$.

\subsection{The implementation}\label{app:float-implementation}

For $d=2$ and a control $(p,q)\neq(0,0)$, the empty-stack recurrence
\eqref{eq:G} reads, in the reduced entries,
\begin{equation}\label{eq:G-reduced}
G_{(p,q)}=\sum_{h=0}^{p-1}G_{(h,p+q-1-h)}
 +\one_{q\geq1}G_{(p,q-1)}
 +\sum_{h=0}^{q-2}\sum_{c=0}^{p}\sum_{t\geq0}R_{q-1-h,p-c}(h,t)\,G_{(c,t)}.
\end{equation}
The middle term is the last-band move with $\delta_h=0$.  At $N=300$ the
largest values are about $2^{1115}$, beyond the largest binary64 number,
which is below $2^{1024}$.  The implementation therefore stores the scaled
values
\[
 R'_{\ell,a}(q,s)=2^{-2(\ell+a+q)}R_{\ell,a}(q,s),
 \qquad
 G'_{(p,q)}=2^{-2(p+q)}G_{(p,q)}.
\]
In the scaled values, every band term and every $D_{\ell,a}$ with
$\ell\geq2$ in \eqref{eq:R-recurrence} is multiplied by $1/4$, the split
term becomes
$4^{m-1}R'_{\ell_1,a_1}(q,m)R'_{\ell_2,a_2}(m,s)$, and $D_{1,a}(q,s)$ becomes
$2^{-2(1+a+q)}\one_{a=0,\,q=s}$.  In \eqref{eq:G-reduced} the first two terms
are multiplied by $1/4$ and the last becomes
$4^{t-1}R'_{q-1-h,p-c}(h,t)G'_{(c,t)}$.  In the build for $N=300$ the nonzero
scaled values lie between $2^{-600}$ and $2^{515}$, inside the normal range
of binary64, and none is infinite or subnormal.

A term $4^{m-1}r_1r_2$ with scaled factors $r_1,r_2$ is formed as
$(2^{m-1}r_1)(2^{m-1}r_2)$, or as $(4^{m-1}r_1)r_2$ when
$4^{m-1}r_1\leq10^{300}$.  The factors are then at least $2^{-603}$, while
$r_1r_2$ can be as small as $2^{-1200}$ and would underflow.  The program computes the reduced
entries in increasing reduced grade and then the values $G'_{(p,q)}$ in
increasing $p+q$.  It forms each value by adding its summands to a running
sum, one at a time, in a fixed order, starting from zero.  In
\eqref{eq:G-reduced} the terms with fixed $h$ and $c$ are first added into
an inner sum over $t$.  The programs use IEEE binary64 arithmetic with
rounding to nearest and are compiled without options, such as
\texttt{-ffast-math}, that allow the compiler to reassociate or otherwise
reorder floating-point operations.  The compiler may fuse a multiplication
with the addition that follows it.

The sampler follows \cref{prop:sampling}.  Its state is the control $(p,q)$
together with the stack of blocks, each carrying its size and the
control at which it must be exposed.  At an empty-stack state the candidates
of a letter choice are the summands of \eqref{eq:G-reduced}, and the stored
value of the state is $G'_{(p,q)}$.  At a state whose head has size $\ell$
and is to be exposed at control $(c,t)$, the stored value is
$R'_{\ell,a}(q,t)$ with $a=p-c$, and the candidates are the summands of
\eqref{eq:R-recurrence}.  For $\ell\geq2$ the term $D_{\ell,a}(q,t)$ is two
candidates, one for each endpoint of the head.  Each candidate is given its
value in the scaled recurrence, so that the exact values of the candidates
of a state sum to the exact scaled value of the state.

The candidates are scanned in a fixed order, in groups.  Each band term and
each endpoint is a group of its own.  At an empty-stack state the terms of
the last sum of \eqref{eq:G-reduced} with fixed $h$ and $c$ form one group,
and inside a head the split terms with fixed $\ell_1$ and $a_1$ form one
group.  The sampler computes the sum of each group by adding its members in
order, and the running sums $A_1\leq A_2\leq\cdots$ of the group sums by
adding these in order.  It takes a multiple $x$ of $2^{-53}$ in $[0,1)$ from
$53$ random bits and rounds $x\hat T$ to $y$, where $\hat T$ is the stored
value of the state.  The random bits come from the generator xoshiro256**
of Blackman and Vigna~\cite{BlackmanVigna2021}, seeded for each sample
through splitmix64.  The sampler then takes the first group $g$ with
$A_g>y$ and,
inside that group, the first member whose running sum exceeds the rounded
difference $y-A_{g-1}$, or the last member if there is none.  Candidates of
value zero are skipped.  If no group has $A_g>y$, the sampler scans again
with $y$ multiplied by $\hat A/\hat T$, where $\hat A$ is the total of the
scan, and if
that also falls short it takes the last candidate of positive value.

\subsection{The error bound}\label{app:float-bound}

We write $\mathrm{fl}(z)$ for $z$ rounded to binary64, $u=2^{-53}$ for the
unit roundoff and $\gamma_m=mu/(1-mu)$ for $mu<1$.  If $z$ is the sum or the
product of two binary64 numbers, or a product plus a number in a fused
operation, and $|z|$ lies in the normal range, then
$\mathrm{fl}(z)=z(1+\delta)$ with $|\delta|\leq u$.  A multiplication by a
power of two is exact when the result is normal.  If $|\theta_j|\leq\gamma_j$
and $|\theta_k|\leq\gamma_k$, then $(1+\theta_j)(1+\theta_k)=1+\theta$ with
$|\theta|\leq\gamma_{j+k}$, and a product of $m$ factors $1+\delta_i$ with
$|\delta_i|\leq u$ is $1+\theta$ with $|\theta|\leq\gamma_m$
\cite[Lemmas~3.1 and~3.3]{Higham2002}.

For $R_{\ell,a}(q,s)$ put $\lambda=\ell+a+q-s$, and for $G_{(p,q)}$ put
$\lambda=p+q$.  This is the number of letters read by the stopped paths of
$R_{\ell,a}(q,s)$ and by the complete paths from the state
$((p,q),\varnothing)$, and
$\lambda\leq N$.  Every band term and every $D$ term has one letter fewer.
The two factors of a split term together have $\lambda-1$ letters, since the
split reads one letter itself, and the same holds for the two factors of a
term of the last sum of \eqref{eq:G-reduced}.  Put $\nu=3\,352\,801$.
We count the summands of \eqref{eq:G-reduced}, and of
\eqref{eq:R-recurrence} for fixed $s$, with every index in the support of
\cref{lem:d2-exact-support}, counting $D_{\ell,a}$ with $\ell\geq2$ twice,
once for each endpoint.  For $G_{(p,q)}$ with $n=p+q$ there are $n$ of them
if $q=0$, and $n+(n-2)(p+1)(q-1)/2\leq N+N^2(N-2)/8=\nu-1$ if $q\geq1$,
since $(p+1)+(q-1)=n$, with equality at $(149,151)$.  For $R_{\ell,a}(q,s)$
with $t=a+q$ there are at most $t+1$ if $\ell=1$, and at most
$t+\ell+(\ell-2)t(t+1)/2\leq N+N(N-2)^2/8<\nu-1$ if $\ell\geq2$, since
the band and $D$ terms number at most $t+2$, the split terms with fixed
$\ell_1$ at most $(a+1)q+a(a+1)/2+1\leq t(t+1)/2+1$, while $t+\ell\leq N$
and $(\ell-2)t\leq(N-2)^2/4$.  So no value has more than $\nu-1$ summands and no
letter choice has more than $\nu-1$ candidates.

\begin{lemma}[Accuracy of the tables]\label{lem:float-tables}
Every stored value is its scaled exact value times $1+\theta$ with
$|\theta|\leq\gamma_{\nu\lambda}$.  In particular, a stored value is zero
if and only if its exact value is zero.
\end{lemma}

\begin{proof}
We argue by induction on $\lambda$, in the order in which the values are
computed.  A nonzero exact scaled value is a positive integer times
$2^{-2w}$, where $w\leq N$ is the reduced grade or $p+q$, so it is at least $2^{-600}$, and so is every
nonzero summand.  Since the summands are nonnegative and rounding is
monotone, every running sum lies between its first nonzero summand and the
stored value.  By the induction hypothesis, a computed product of stored
factors is the exact value of its summand times a number between $1/2$ and
$2$, so it is at least $2^{-601}$.  The build for $N=300$
checks the range of the stored values given in
\cref{app:float-implementation}, and with it and the form of the products
there, every nonzero number formed lies between $2^{-603}$ and $2^{997}$, so
no underflow or overflow occurs.

A summand that is one stored value times a power of two is formed without
rounding.  A product of two stored values is formed with one rounding, or
with none if it is fused with the next addition.  By the induction hypothesis
and since the letters of the factors add up to at most $\lambda-1$, the
exact product of the stored factors of a summand, with its power of two, is
the exact value of the summand times $1+\theta$ with
$|\theta|\leq\gamma_{\nu(\lambda-1)}$.  The summand then
undergoes at most $\nu$ roundings: at most one in the product, at most
one in each addition to the running sum of its group, and at most one in
each addition of a group sum.  Therefore the computed value is
$\sum_i z_i(1+\theta_i)$, where the $z_i\geq0$ are the exact summands and
$|\theta_i|\leq\gamma_{\nu\lambda}$.  This is
$(1+\theta)\sum_i z_i$ with $|\theta|\leq\gamma_{\nu\lambda}$.
\end{proof}

\begin{proposition}[Floating-point sampling]\label{prop:float-sampling}
Let $n\leq300$ and $\bar\gamma=\gamma_{\nu N}$.  Assume that the random
bits used by the sampler are independent and uniform.  Then the output of
the floating-point sampler is within total variation distance
\[
 n\Bigl(\frac{\bar\gamma}{1-\bar\gamma}+7\nu u\Bigr)<3.5\cdot10^{-5}
\]
of the uniform distribution on $\Av_n(12453)$.
\end{proposition}

\begin{proof}
By \cref{lem:float-tables} the sampler reaches only states of positive
count.  Fix such a state, with stored value $\hat T$ and candidates
$1,\ldots,k$ in scan order.  Let $w_i$ be the scaled exact value of
candidate $i$ and $W=w_1+\cdots+w_k$.  Let $v_i$ be the exact product of the
stored values in candidate $i$ and its power of two, and let
$\hat w_i=\mathrm{fl}(v_i)$ be the value the sampler computes, with
$\hat w_i=v_i$ when there is one factor.

\emph{The values.}  By \cref{lem:float-tables} applied to the factors,
$\hat w_i=w_i(1+\theta_i)$ with
$|\theta_i|\leq\gamma_{\nu(\lambda-1)+1}\leq\bar\gamma$, and
$\hat T=W(1+\theta)$ with $|\theta|\leq\bar\gamma$.  Therefore
\begin{equation}\label{eq:float-values}
 \sum_{i=1}^k\Bigl|\frac{\hat w_i}{\hat T}-\frac{w_i}{W}\Bigr|
 =\sum_{i=1}^k\frac{w_i}{W}\Bigl|\frac{1+\theta_i}{1+\theta}-1\Bigr|
 \leq\frac{2\bar\gamma}{1-\bar\gamma}.
\end{equation}
The stored value $\hat T$ and the total $\hat A$ of the scan are both
computed sums of $v_1,\ldots,v_k$ in which each $v_i$ undergoes at most
$\nu$ roundings.  The two endpoint candidates are one summand of the
table, of value $v_i+v_{i+1}$.  Therefore $\hat T$ and $\hat A$ are both
within a factor $1\pm\gamma_\nu$ of $v_1+\cdots+v_k$, and
$|\hat A-\hat T|\leq2.01\nu u\hat T$.

\emph{The scan.}  For $y<\hat A$ the chosen candidate is a nondecreasing
function of $y$, since rounding is monotone.  So candidate $i$ is chosen for
$y$ in an interval $[\xi_{i-1},\xi_i)$, where
$0=\xi_0\leq\cdots\leq\xi_k=\hat A$.  Consider a group with running sums
$\alpha_1\leq\cdots\leq\alpha_r$ of its members, and let $A_{g-1}$ and $A_g$
be the running sums before and after it.  The group occupies
$[A_{g-1},A_g)$, and its member $j<r$ ends where
$\mathrm{fl}(y-A_{g-1})$ reaches $\alpha_j$, or at $A_g$ if that comes first.
Since $|\mathrm{fl}(z)-z|\leq u|z|$, this breakpoint differs from
$A_{g-1}+\alpha_j$ by at most $1.01u\hat A$.  The difference
$\alpha_j-\alpha_{j-1}$ differs from the value of member $j$ by at most
$u\hat A$, and $A_g-A_{g-1}$ differs from $\alpha_r$ by at most $u\hat A$.
Therefore
\begin{equation}\label{eq:float-scan}
 |\xi_i-\xi_{i-1}-\hat w_i|\leq4.02u\hat A\qquad(1\leq i\leq k).
\end{equation}

\emph{The draw.}  For $\tau\geq0$ let $F(\tau)$ be the probability that
$y\geq\tau$, and let $f(\tau)=\max(0,1-\tau/\hat T)$.  Since
$x\hat T<\hat T$, we have $y\leq\hat T$, so $F(\tau)=0=f(\tau)$ for
$\tau>\hat T$.  For $\tau\leq\hat T$, we have $y\geq\tau$ if
$x\hat T\geq\tau/(1-u)$ and only if $x\hat T\geq\tau/(1+u)$.  Since $x$ is
uniform on the multiples of $2^{-53}$ in $[0,1)$, the probability that
$x\geq z$ lies between $1-z-u$ and $1-z$.  Together,
\begin{equation}\label{eq:float-draw}
 |F(\tau)-f(\tau)|\leq2.01u\qquad(0\leq\tau\leq\hat A).
\end{equation}

\emph{One letter choice.}  Let $\hat p_i$ be the probability that the
sampler chooses candidate $i$.  Without a rescan, candidate $i$ is chosen
with probability $F(\xi_{i-1})-F(\xi_i)$.  By \eqref{eq:float-draw},
replacing $F$ by $f$ changes the sum over $i$ of these probabilities by at
most $4.02ku$ in total.  The differences $f(\xi_{i-1})-f(\xi_i)$ equal
$(\xi_i-\xi_{i-1})/\hat T$ except above $\hat T$, which changes the sum
by at most $\max(\hat A-\hat T,0)/\hat T\leq2.01\nu u$.  By
\eqref{eq:float-scan}, replacing $\xi_i-\xi_{i-1}$ by $\hat w_i$ changes
it by at most $4.03ku$.  A rescan happens with probability
$F(\hat A)\leq f(\hat A)+2.01u\leq2.01\nu u+2.01u$.  Since $k<\nu$,
\[
 \sum_{i=1}^k\Bigl|\hat p_i-\frac{\hat w_i}{\hat T}\Bigr|
 \leq8.05ku+4.02\nu u+2.01u\leq14\nu u.
\]
With \eqref{eq:float-values}, the total variation distance between the
choice of the sampler and the choice with probabilities $w_i/W$ is at most
$\bar\gamma/(1-\bar\gamma)+7\nu u$.

\emph{The sample.}  With probabilities $w_i/W$, the choices produce the
uniform distribution on $\Av_n(12453)$, by \cref{prop:sampling}.  Run the
two processes together, and as long as they are in the same state, couple
their next choices so that they differ with probability equal to the total
variation distance between them.  A sample is produced by $n$ choices, so
the two outputs differ with probability at most
$n(\bar\gamma/(1-\bar\gamma)+7\nu u)$.  At $n=300$ this is
$3.43\cdot10^{-5}$.
\end{proof}

The bound of \cref{lem:float-tables} is far from the observed error: at
$n=150$ it gives about $6\cdot10^{-8}$ for $G_{(n,0)}$, while the
floating-point values $G_{(n,0)}$ agree with the exact coefficients
through $n=150$ to within $3.2\cdot10^{-13}$ relative error.

\section{The Lean development}\label{app:lean}

This appendix describes the Lean development of \cref{sec:formal}.

\subsection{Conventions}

The development works with words over $\mathbb N$ and $0$-based values.  It
defines pattern containment as the existence of an order-isomorphic
subsequence, a finite set of $0$-based words in bijection with $\Av_n(\tau)$,
the one-threshold recurrence $W$ of
\eqref{eq:W} and \eqref{eq:W-boundary}, the literal recurrence $H$ of
\eqref{eq:H} with $H_{\mathbf0}(\varnothing)=1$ for $d=2$, and the
kernel recurrences \eqref{eq:scalar-K}--\eqref{eq:scalar-G} for $d=1$ and
\eqref{eq:K}--\eqref{eq:G} for $d=2$.  In the two kernel systems every sum
over intermediate or terminal controls is restricted to the controls of mass
at most that of the source control of its kernel row (for $d=1$, to
$u\leq p$ in \eqref{eq:scalar-K} and to $t\leq h$ in \eqref{eq:scalar-G}),
which \cref{lem:scalar-support,lem:support} show loses nothing.
In \path{Challenge.lean} the four main theorems of \cref{sec:formal} are
named \path{Av12453.count_1342_literal},
\path{Av12453.count_1342_kernel}, \path{Av12453.count_12453_literal},
and \path{Av12453.count_12453_kernel}.  They restate for Mathlib's
permutations the development's theorems \texttt{av1342\_count},
\texttt{av1342\_count\_kernel}, \texttt{av12453\_count}, and
\texttt{av12453\_count\_kernel}, which are stated for words.

\subsection{Organization}

The Lean project \path{formal/Av12453/} builds two libraries from one
source tree.  The library \path{PermPatterns/} holds the generic apparatus
(words, containment and avoidance, standardization, the patterns $231$ and
$\iota_d$, sums, symmetries, and the passage to Mathlib's permutations of
$\{0,\ldots,n-1\}$), and nothing in it refers to the family $\beta_d$.  The
library \path{Av12453/} holds the development specific to this paper.  The
file \path{Av12453/Axioms.lean}, which neither library imports, checks that
every declaration of the two libraries depends only on propositional extensionality, choice, and quotient
soundness.  The file \path{formal/README.md} describes the modules, the
checks run by \path{check.sh}, and the trusted base.  The
Lean toolchain (version 4.35.0-rc2) and the Mathlib revision are pinned in
the project files.

\subsection{What must be read}

The file \path{Challenge.lean} states the four theorems for Mathlib's
permutations of $\{0,\ldots,n-1\}$, with pattern containment
\texttt{PermContains} and the patterns written as products of transpositions,
together with the definitions of \texttt{W}, \texttt{H}, \texttt{K}, and
\texttt{G}, and \path{comparator.sh} checks that these definitions are identical
to those of the development.  The recurrences are written with a recursion
bound.  The definitions
\texttt{W}, \texttt{H}, \texttt{K}, and \texttt{G} evaluate an auxiliary
function at a bound exceeding, respectively, the number of unread values,
the number of unread values, the grade, and the control mass of the
argument, and the development
proves that every sufficient bound gives the same value.  So the reader need
only compare
the equations of the auxiliary functions with the displays \eqref{eq:W} and
\eqref{eq:W-boundary}, \eqref{eq:H} with $H_{\mathbf0}(\varnothing)=1$,
\eqref{eq:scalar-K}--\eqref{eq:scalar-G}, and \eqref{eq:K}--\eqref{eq:G}.
\path{Challenge.lean} writes the terms with $a=0$ or $b=0$ of the split sums
of these displays as one endpoint term, with a case distinction on whether
the head has size one, which is the same sum.  The development's restated
kernel equations call this term \texttt{D}, and those of \texttt{W} and
\texttt{H} call it \texttt{Eend}.  The opening docstring of
\path{Challenge.lean} gives this dictionary, with the $0$-based conventions
spelled out, and the displays are also restated as proved theorems
(\texttt{K\_succ\_eq}, \texttt{G\_eq}, and their analogues), which serve as
reading aids but are not part of what must be trusted.

\subsection{Proofs and checks}

The thresholds $b_1,b_2$ are defined as the least final values of
increasing subsequences, and the interval stack by the moves of
\cref{prop:state-invariant}(b), which classify a letter by the band or
interval containing it, so \cref{lem:separators} and
\cref{prop:scan-states}(a) and (b) are proved for this notion of legal
prefix.  \Cref{lem:trigger,lem:first-letter} are formalized, but the four
main theorems do not use them.  Since the kernels are defined by their
recurrences, \cref{prop:kernel-recurrence} is not needed.  The restriction of the
control sums is proved lossless, and it is also what makes recursion on
the grade terminate for \eqref{eq:scalar-K} and \eqref{eq:K} read as
definitions.

The development has over eleven thousand lines in thirty modules.  The
definitions were also
exercised by kernel evaluation: the scans of the running example
\eqref{eq:running-example} reproduce the rows of
\cref{ex:1342-scan,ex:full-state}, kernel evaluation of the recurrences
reproduces the first seven terms of \eqref{eq:first-terms}, and compiled
evaluation the first ten.

\bibliographystyle{amsplain}
\bibliography{av12453_polytime}

\end{document}